\documentclass[11pt]{article}
\usepackage{amsmath}
\usepackage{amsfonts}
\usepackage{amssymb}
\usepackage[english]{babel}
\usepackage{graphicx}
\usepackage{amsthm}
\usepackage{hyperref}
\usepackage{accents}
\usepackage{empheq}
\usepackage [autostyle, english = american]{csquotes}
\MakeOuterQuote{"}
\numberwithin{equation}{section}
\theoremstyle{plain}
\newtheorem*{theorem*}{Theorem}
\newtheorem*{lemma*}{Lemma}
\newtheorem{theorem}{Theorem}
\newtheorem{lemma}{Lemma}[section]
\newtheorem{corollary}[lemma]{Corollary}

\newtheorem{innercustomthm}{Theorem}
\newenvironment{customthm}[1]
  {\renewcommand\theinnercustomthm{#1}\innercustomthm}
  {\endinnercustomthm}

\theoremstyle{definition}
\newtheorem{definition}[lemma]{Definition}
\newtheorem{remark}[lemma]{Remark}

\renewcommand{\div}{{\rm div}}

\def\pr{\partial}
  
\def\V{\Vert}

\usepackage{geometry}
\begin{document}

\title{Finite-time Singularity formation for Strong Solutions to the axi-symmetric $3D$ Euler Equations}
\author{Tarek M. Elgindi\textsuperscript{1} and In-Jee Jeong\textsuperscript{2}}
\footnotetext[1]{Department of Mathematics, UC San Diego. E-mail: telgindi@ucsd.edu.}
\footnotetext[2]{Department of Mathematics, Korea Institute for Advanced Study. E-mail: ijeong@kias.re.kr.}
\date{\today}
\maketitle

\begin{abstract}
For all $\epsilon>0$, we prove the existence of finite-energy strong solutions to the axi-symmetric $3D$ Euler equations on the domains $  \{(x,y,z)\in\mathbb{R}^3: (1+\epsilon|z|)^2\leq x^2+y^2\}$ which become singular in finite time. We further show that solutions with 0 swirl are necessarily globally regular. The proof of singularity formation relies on the use of approximate solutions at exactly the critical regularity level which satisfy a $1D$ system which has solutions which blow-up in finite time. The construction bears similarity to our previous result on the Boussinesq system \cite{EJB} though a number of modifications must be made due to anisotropy and since our domains are not scale-invariant. This seems to be the first construction of singularity formation for finite-energy strong solutions to the actual $3D$ Euler system. 
\end{abstract}

\tableofcontents

\section{Introduction}

The problem of finite-time singularity formation for solutions to the $3D$ Euler equations is one of the classical problems in the study of PDEs and has stood the test of time for over two centuries. Though the problem has remained open until now, there have been several fantastic advancements by many authors -- especially over the last few years. The goal of this work is to give a new take on the problem which allows us to prove finite-time singularity formation for the $3D$ Euler equations in the "critical" setting which clearly sets the $3D$ Euler equations apart from the $2D$ Euler equations and similar models. We begin by recalling the incompressible Euler equations and their salient features. 

\subsection{The $3D$ Euler Equations}
Recall the $n$-dimensional incompressible Euler equations:
\begin{empheq}[left = \empheqlbrace]{align}
\label{E} &\partial_t u+u\cdot\nabla u+\nabla p=0,
\\
\label{divfree} &\div(u)=0,
\end{empheq} for the velocity field $u: \mathbb{R} \times \Omega \rightarrow \mathbb{R}^n$ and internal pressure $p:\mathbb{R}\times\Omega \rightarrow \mathbb{R}$ of an ideal (frictionless) fluid. This system models the evolution of the velocity field of an ideal fluid through a suitable closed subset $\Omega\subset\mathbb{R}^n$. We also impose the no penetration boundary condition $u\cdot n=0$ where $n$ is the outer normal at the boundary of $\Omega$. 
We also supply the system with a divergence-free initial datum $u_0$, which is the velocity field at time $t=0$. 

It is well known that for any $u_0$ which is sufficiently smooth, there exists a unique local-in-time solution $u$ to \eqref{E} -- \eqref{divfree} with $u|_{t=0}=u_0$. The amount of smoothness which is required to establish existence and uniqueness roughly corresponds to the amount of smoothness required to define each quantity in \eqref{E} -- \eqref{divfree} point-wise. We call such solutions strong solutions. It is known that strong solutions conserve energy. Indeed, upon multiplying \eqref{E} by $u$, integrating over $\Omega$, and using \eqref{divfree} and the no-penetration boundary condition $u\cdot n=0$ on $\partial\Omega$, we see:
$$\frac{d}{dt} \int_{\Omega} |u(t,\mathbf{x})|^2d\mathbf{x}=0.$$
Conservation of energy seems to indicate that solutions to \eqref{E} -- \eqref{divfree} cannot grow too much -- though it does not preclude growth of $\nabla u$ or even the sup-norm of $u$. Moreover, for systems like the Euler equations, it is usually necessary to have global point-wise control of $\nabla u$ to prevent finite-time singularity formation. Due to this "regularity gap", the global regularity problem when $n\geq 3$ remains a major open problem in the field\footnote{See, for example, http://www.claymath.org/sites/default/files/navierstokes.pdf.}:

\vspace{3mm}

\begin{centering}
	{\bf Global Regularity Problem:} \,\,\emph{Does there exist a solution $u$ of the $3D$ Euler equations with finite energy such that $u\in C^\infty([0,1) \times \Omega )$ but $\limsup_{t\rightarrow 1} \V\nabla u(t,\cdot)\V_{L^\infty}=+\infty$?}
\end{centering}

\vspace{3mm}

\noindent We will consider the following more general problem:

\vspace{3mm}

\begin{centering}
	{\bf Generalized Global Regularity Problem\footnote{It is important to remark that, to avoid ill-posedness issues (as in \cite{EM1} and \cite{BL2}), it is necessary to ask that $\nabla u$ is bounded on a time-interval and not just at the initial time.  In fact, one could simply ask whether there is a Banach space $X\subset L^2\cap W^{1,\infty}(\mathbb{R}^3)$ where the $3D$ Euler equations are locally well-posed but not globally well-posed.}:} \,\,\emph{Does there exist a solution $u$ of the $3D$ Euler equations with finite energy such that $u\in W^{1,\infty}([0,1)\times \Omega)$ but $\limsup_{t\rightarrow 1} \V \nabla u(t,\cdot)\V_{L^\infty}=+\infty$?}
\end{centering}

\vspace{3mm}

\noindent The purpose of this work is to answer this question for the domains $\Omega_{\epsilon}^{3D}=\{(x,y,z): (1+\epsilon|z|)^2\leq (x^2+y^2)\}$ for any $\epsilon>0$ (see Figure \ref{fig:domain}). In fact, we will prove that there is a local well-posedness class $X\subset L^2\cap W^{1,\infty}$ and a local solution $u$ belonging to that class for all $t<1$ for which $\lim_{t\rightarrow 1}\V \nabla u(t,\cdot)\V_{L^\infty}=+\infty.$ Moreover, for that solution, $\partial_t u$, $u\cdot\nabla u$, and $\nabla p$ are all (H\"older) continuous in space-time in   $[0,1)\times\Omega^{3D}_\epsilon$. In this sense, the solution we construct is really a strong solution. 

\subsection{Previous Works}

An important quantity to consider when studying the Euler equations, particularly in two dimensions, is the vorticity $\omega:=\nabla \times u$. In three dimensions, the vorticity satisfies the following equation:
$$\partial_t\omega+u\cdot\nabla\omega=\omega\cdot\nabla u.$$ The term on the right hand side is called the vortex stretching term due to its ability to amplify vorticity. Due to the fact that $\div(u)=0$, it is actually possible to recover $u$ from $\omega$ and the map $\omega\mapsto u$ is called the Biot-Savart law. In terms of regularity, $\omega$ and $\nabla u$ are comparable; however, the difficulty is to understand the geometric properties of $\omega\cdot\nabla u$. Neglecting these geometric properties leads one to believe that $\omega\cdot\nabla u\approx \omega|\omega|$ so that singularity formation is trivial. On the other hand, in the $2D$ case, when $u=(u_1(x,y),u_2(x,y),0)$, it is easy to see that $\omega\cdot\nabla u\equiv 0$ which then leads to global regularity. A good understanding of the vortex stretching term and its interaction with the transport term $u\cdot\nabla \omega$ is necessary to make progress on the global regularity problem. 

We now collect a few of the important works on the $3D$ Euler equations and related models. The relevant literature is quite vast so we will focus only on four general areas: local well-posedness and continuation criteria, model problems, weak solutions, and numerical works. 
\subsubsection{Local well-posedness and Blow-up Criteria}
The existence of local strong solutions to \eqref{E} -- \eqref{divfree} in two and three dimensions is classical and goes back at least to Lichtenstein in 1925 \cite{Lich30} who proved local well-posedness of finite-energy solutions in the H\"{o}lder spaces $C^{k,\alpha}$ for any $k\in \mathbb{N}$ and $0<\alpha<1$. Kato \cite{Kato86} established local well-posedness for velocity fields in the Sobolev spaces $H^s$ for $s>\frac{n}{2}+1$. The restrictions on $\alpha$ and $s$ in the results of Lichtenstein and Kato were shown to be sharp in (\cite{BL1}, \cite{BL2}, \cite{EM1}, and \cite{EJ}). Local well-posedness in Besov spaces was established by Vishik \cite{Vi1} and Pak and Park \cite{PakPark}. The most well-known blow-up criterion is that of Beale, Kato, and Majda \cite{BKM} which states that a $C^{1,\alpha}$ or $H^s$ solution ($0<\alpha<1$ and $s>\frac{n}{2}+1$) loses its regularity at $T^*$ if and only if $\limsup_{t\rightarrow T^*}\int_0^t \V \omega(s)\V_{L^\infty}ds=+\infty.$ In fact, this result is the motivation for the "Generalized Global Regularity Problem" above. Another important blow-up criterion is that of Constantin, Fefferman, and Majda \cite{CFM96} which roughly says that if the vorticity has a well-defined "direction" and if the velocity field is uniformly bounded, then the $3D$ Euler solution looks like a $2D$ Euler solution and no blow-up can occur. Improvements on these criteria were given in \cite{KozonoTaniuchi} and \cite{JianHouYu}. 

\subsubsection{Model Problems}
Over the years, a number of model problems have been introduced and analyzed to better understand the dynamics of solutions to the $3D$ Euler equations. One such $1D$ model was studied by Constantin, Lax, and Majda \cite{CLM} in 1985 where the vorticity equation was replaced by a simple non-local equation which modeled only the effects of vortex stretching. In the same paper, the authors established singularity formation for a large class of data \cite{CLM}. Thereafter, De Gregorio \cite{DG1} and then Okamoto-Sakajo-Wunsch \cite{OSW} introduced generalizations of the work \cite{CLM}. It turns out that these models are almost identical to the equation for scale-invariant solutions to the SQG system which we studied in \cite{EJSI} and \cite{EJDG}. After important numerical and analytical works of Hou and Luo \cite{HouLuo} and Kiselev and \v{S}ver\'ak \cite{KS} respectively, a new class of models was introduced to study the axi-symmetric $3D$ Euler equations near the boundary of an infinite cylinder (see \cite{CKY} and \cite{HLModel}). Another model, coming from atmospheric science, which has gained much attention is the surface quasi-geostrophic (SQG) system (\cite{CMT1}, \cite{CMT2}, \cite{Majda03}) which can be seen as a more singular version of the $2D$ Euler equations and a good model of the $3D$ Euler equations. Global regularity for strong and smooth solutions to the SQG equation is still wide open though substantial progress on the problem of "patch" solutions has been made in \cite{KRYZ}. We also mention some of the works on shell-models where the Euler system on $\mathbb{T}^3$ is seen as an infinite system of ODE and then all interactions except a few are neglected; in several cases, blow-up for these models can be derived. See the works of Katz-Pavlovic \cite{KP}, Friedlander-Pavlovic \cite{FP}, Kiselev-Zlatos \cite{KZ2}, and Tao (\cite{TaoEuler} and \cite{TaoManifold}).  Of note is that Tao \cite{TaoManifold} recently showed that any finite-dimensional bilinear and symmetric ODE system with a certain cancellation property can be embedded into the incompressible Euler equations on some  (high dimensional) compact Riemannian manifold. 

Closer to the actual $3D$ Euler equations are a model introduced by Hou and Lei \cite{HouLei} which is the same as the axi-symmetric $3D$ Euler equations without the transport term (see the next subsection). Singularity formation for this model is conjectured in \cite{HouLei} though it seems to still be open in settings where solutions have a coercive conserved quantity. 

\subsubsection{Weak Solutions}

One of the reasons that the global regularity problem for strong solutions to the $3D$ Euler equations is important is that there is no good theory of weak solutions available--even\footnote{While Yudovich \cite{Y1} solutions are usually called weak solutions, we feel that classifying them as such is slightly misleading in the present context. Besides, the Yudovich theory does not extend to $3D$ even locally in time.} in $2D$. In fact, weak solutions have been shown to exhibit very wild behavior such as non-uniqueness and non-conservation of energy. See, for example, the works of Scheffer \cite{Scheffer}, Shnirelmann \cite{Shnir}, De Lellis-Sz\'{e}kelyhidi (\cite{DS1} and \cite{DS2}), and more recently Isett \cite{Isett} and Buckmaster-De Lellis-Sz\'{e}kelyhidi-Vicol \cite{BDSV}. We also mention that Kiselev and Zlatos \cite{KZ} have shown that in a domain with cusps, the $2D$ Euler equations can blow up in the sense that initially continuous vorticity may become discontinuous in finite time.

\subsubsection{Previous blow-up results for infinite-energy solutions}

We should mention that there have been a number of infinite-energy solutions to the actual $2D$ and $3D$ Euler equations which have been shown to become singular in finite time. The well-known ``stagnation-point similitude'' ansatz (which goes back to the work of Stuart \cite{St} in 1987) takes the following form in $3D$: \begin{equation}\label{eq:stag}
u(t,x,y,z) = (u_1(t,x,y),u_2(t,x,y),z \gamma(t,x,y)). 
\end{equation} Constantin \cite{Con} has shown that smooth initial data of the form \eqref{eq:stag} may blow up in finite time, shortly after numerical simulations by Gibbon and Ohkitani \cite{OG}.  In the $2D$ case, one can take $u_2 \equiv 0$ and $u_1, \gamma$ be independent of $y$. Blow-up in this case was shown even earlier (see \cite{St}, \cite{CISY}).
Similarly, in the usual cylindrical coordinates $(r,\theta,z)$, one may consider the following ansatz (Gibbon-Moore-Stuart \cite{GMS}): \begin{equation}\label{eq:stag2}
u(t,r,\theta,z) = u^r(t,r,\theta)e^r + u^\theta (t,r,\theta)e^\theta + z\gamma(t,r,\theta) e^z ,
\end{equation} where $u^r$ and $u^\theta$ respectively denote the radial and swirl component of the velocity. The authors in \cite{GMS} have found simple and explicit solutions having the form \eqref{eq:stag2} which blow up in finite time. The blow-up is present even in the no-swirl case. Blow-up for closely related systems was shown using similar ansatz (see for instance Gibbon-Ohkitani \cite{GO1} and Sarria-Saxton \cite{SarriaSaxton}). 

Note that in all these examples, the vorticity is never a bounded function in space (indeed, it grows linearly at infinity), and it is unclear whether the dynamics of such solutions are well-approximated by finite-energy solutions. Of course, such a statement cannot be valid in the $2D$ Euler or no-swirl axisymmetric setting.  We will also introduce a class of infinite-energy \emph{approximate} solutions. However, since we base ours on scale-invariance and symmetry, they have globally bounded vorticity (before blow-up) and also are well approximated by compactly supported solutions; in particular, they are globally regular in the $2D$ case \cite{EJSI}. 

\subsubsection{Numerical Works}

It is impossible to do justice to the vast literature on numerical studies of the $3D$ Euler equations.  We refer the reader to the survey papers of Gibbon \cite{Gib} and Gibbon, Bustamante, and Kerr \cite{GibBuKe} for an extensive list of numerical works on the $3D$ Euler equations. In the simulations of Pumir and Siggia \cite{PS} dating back to 1992, a $10^6$ increase in vorticity was observed in the axi-symmetric setting. Also very well known are numerical results using perturbed antiparallel vortex tubes by Kerr (\cite{Ker1}, \cite{Ker2}), which suggested finite time blow-up of the vorticity. For a further discussion as well as more refined simulations on Kerr's scenario, see Hou and Li \cite{HouLi} and Bustamante and Kerr \cite{BuKe}. We wish to also make mention of more recent works of Luo and Hou (\cite{HouLuo}, \cite{HL}) where very large amplification of vorticity is shown for some solutions to the axi-symmetric $3D$ Euler equations in an infinite cylinder. Luo and Hou's paper was the motivation for a number of recent advances in this direction, including this work.  In fact, the reader may notice that the spatial domains we consider here, $\{(x,y,z): (1+\epsilon|z|)^2\leq (x^2+y^2)\}$ for $\epsilon>0$ is very similar to the setting of \cite{HouLuo} (except that our domain is the \textit{exterior} of a cylinder). We also mention a recent interesting work of Larios, Petersen, Titi, and Wingate \cite{LPTW} where singularity in finite time is observed for spatially periodic solutions to the $3D$ Euler equations.
We end this discussion with a quote from J. Gibbon regarding the finite-time singularity problem: "\emph{Opinion is largely divided on the matter with strong positions taken on each side. That the vorticity accumulates rapidly from a variety of initial conditions is not under dispute, but whether the accumulation is sufficiently rapid to manifest singular behaviour or whether the growth is merely exponential, or double exponential, has not been answered definitively.}"

\subsection{Symmetries for the $3D$ Euler Equations}
We now move to discuss the present work and its theoretical underpinnings: rotational and scaling symmetries.  
It is well known that solutions to many of the canonical equations of fluid mechanics satisfy certain scaling and rotational symmetries. A common tool used in many different settings in PDE is to restrict the class of solutions to those which are invariant with respect to some or all of those symmetries. This usually allows one to reduce the difficulty of the problem at hand. For example, in many multi-dimensional evolution equations, it is commonplace to consider spherically symmetric data to reduce a given PDE to a 1+1 dimensional problem. This point of view has also been adopted in the study of the incompressible Euler equations. Indeed, recall that if $\lambda\in \mathbb{R}-\{0\}$ and $\mathcal{O}\in O(n)$, the orthogonal group on $\mathbb{R}^n$ and if $u(t,\cdot)$ is a solution to the incompressible Euler equations, then $\frac{1}{\lambda}u(t,\lambda\cdot)$ and $\mathcal{O}^T u(t,\mathcal{O}\cdot)$ are also solutions. Schematically, we may write this as:
If $$u_0(\cdot)\mapsto u(t,\cdot),$$
$$\frac{1}{\lambda} u_0(\lambda \cdot)\mapsto \frac{1}{\lambda} u(t,\lambda\cdot),$$
$$\mathcal{O}^T u_0(\mathcal{O} \cdot)\mapsto \mathcal{O}^T u(t,\mathcal{O} \cdot)$$ for all $\mathcal{O}\in O(n)$ and $\lambda\in \mathbb{R}-\{0\}$. 
In this sense, we say that the Euler equations satisfies a scaling\footnote{We are aware that the incompressible Euler equations satisfies a two-parameter family of scaling invariances. However, using the time scaling invariance introduces a number of difficulties which are still not fully understood.} and rotational symmetry.

\subsubsection{Rotational Symmetry} It is natural to ask whether one could use the symmetries of the Euler equations to reduce the $3D$ system to a lower-dimensional system with possibly less unknowns. The first attempt may be to  search for solutions which are spherically symmetric, i.e. which satisfy that $\mathcal{O}^Tu(\mathcal{O} \mathbf{x},t)=u(\mathbf{x},t)$ for all $\mathbf{x} = (x,y,z) \in \mathbb{R}^3 $ and all rotation matrices $\mathcal{O}$. Certainly if we had a nice initial velocity field $u_0$ which was spherically symmetric, the solution would formally remain spherically symmetric. Unfortunately, in three dimensions, a spherically symmetric velocity field which is also divergence-free is necessarily trivial for topological reasons. The next attempt, which is classical, is to consider axi-symmetric data. That is, we first pick an axis, such as the $z$-axis, and we then search for solutions which satisfy that $\mathcal{O}^Tu(\mathcal{O}\mathbf{x},t)=u(\mathbf{x},t)$ for all $\mathbf{x}$ and all rotation matrices $\mathcal{O}$ which fix the $z$-axis. This allows one to reduce the full $3D$ Euler system to a two-dimensional system with two components, $u^\theta$ and $\omega^\theta,$ called the swirl velocity and axial vorticity respectively (\cite{MB}): \begin{equation}\label{eq:Majda1}
\begin{split}
\frac{\tilde{D}}{Dt}\Big(\frac{\omega^\theta}{r}\Big)=\frac{1}{r^4} \partial_{z}[(ru^\theta)^2], \qquad \frac{\tilde D}{Dt}\Big(r u^\theta\Big)=0,
\end{split}
\end{equation} supplemented with \begin{equation}\label{eq:Majda2}
\begin{split}
\frac{\tilde D}{Dt}=\partial_t + u^r\partial_r+ u^z \partial_{z}, \qquad u^r=\frac{\partial_{z}\psi}{r}, \qquad u^z=-\frac{\partial_r\psi}{r},
\end{split}
\end{equation} and \begin{equation}\label{eq:Majda3}
\begin{split}
\tilde{L}\psi =\frac{\omega^\theta}{r},\qquad \tilde{L}=\frac{1}{r}\partial_r (\frac{1}{r} \partial_r)+\frac{1}{r^2}\partial_{z}^2,
\end{split}
\end{equation} 
where $r=\sqrt{x^2+y^2}$.

Once $\omega^\theta$ and $u^\theta$ are known, the above system closes. Indeed, $\omega^\theta$ determines $\psi$ through inverting the operator $L$ and $u^r$ and $u^z$ are determined from $\psi$. Dynamically, the axial vorticity $\omega^\theta$ produces a velocity field $(u^r, u^z)$ in the $r$ and $z$ directions which advects the swirl $u^\theta$. Then a derivative of the swirl component forces the axial vorticity. It is conceivable that strong advection of $u^\theta$ causes vorticity growth and this vorticity growth causes stronger advection and that uncontrollable non-linear growth occurs until singularity in finite time. Getting hold of this mechanism requires strong geometric intuition and, seemingly, much more information than what was known about the system. This scenario of vorticity enhancement by the derivative of an advected quantity is precisely the situation in the $2D$ Boussinesq system which we studied in \cite{EJB}. To get hold of this mechanism, we will further restrict our attention to solutions which are locally scale invariant. 

\subsubsection{Scaling Symmetry}

Using the rotational symmetry, we have passed from the full $3D$ Euler system to the axi-symmetric $3D$ Euler system which is a $2D$ system. We will now explain how to reduce the $3D$ Euler system to a $1D$ system by considering asymptotically scale-invariant data. Let us first define the axi-symmetric domains $\Omega^{3D}_\epsilon$ for $\epsilon>0$ by $$\Omega^{3D}_\epsilon:= \{(x,y,z): (1+\epsilon|z|)^2\leq (x^2+y^2)\}=\{(r,z):  1+  \epsilon|z| \leq r\}.$$ In $r,z$ coordinates, $\Omega^{3D}_\epsilon$ is just a sector with its tip at $(r,z)=(1,0)$. Let us further pass to $(\eta,z)$ coordinates where $r=\eta+1$. Thus, $\Omega_\epsilon^{3D}$, in these coordinates, is just $\{(\eta,z): \epsilon|z|\leq \eta\}.$   Now let us see how the axi-symmetric $3D$ Euler equation looks in these coordinates. Since all we have done is shift in $r$, we get:
$$ \frac{\tilde{D}}{Dt}\Big(\frac{\omega^\theta}{\eta+1}\Big)=\frac{1}{(\eta+1)^4} \partial_{z}[((\eta+1)u^\theta)^2], \qquad \frac{\tilde D}{Dt}\Big((\eta+1) u^\theta\Big)=0,$$ where
$$\frac{\tilde D}{Dt}=\partial_t + u^r\partial_\eta+ u^z \partial_{z}, \qquad u^r=\frac{\partial_{z}\psi}{\eta+1}, \qquad u^z=-\frac{\partial_\eta\psi}{\eta+1},$$
$$\tilde{L}\psi =\frac{\omega^\theta}{\eta+1},\qquad \tilde{L}=\frac{1}{\eta+1}\partial_\eta (\frac{1}{\eta+1} \partial_\eta)+\frac{1}{(\eta+1)^2}\partial_{z}^2.$$

Our goal will be to produce a solution which is concentrated near $z=\eta=0$ which both belongs to a local well-posedness class and which becomes singular in finite time. Since we are localizing near $z=\eta=0$ we are led to formally set $\eta=0$ wherever $\eta$ shows up explicitly in the equation. We are then led to the system:
$$ \frac{\hat{D}\omega^\theta}{Dt}=2u^\theta\partial_{z}{u^\theta}, \qquad \frac{\hat Du^\theta}{Dt}=0,$$ with 
$$\frac{\hat D}{Dt}=\partial_t + \hat u^r\partial_\eta+ \hat u^z \partial_{z}, \qquad \hat u^r=\partial_{z}\psi, \qquad \hat u^z=-\partial_\eta\psi,$$ and finally
$$\Delta_{\eta,z}\psi =\omega^\theta.$$ At this point we write: $u^\theta=1+\rho$ and we get:
$$ \frac{\hat{D}\omega^\theta}{Dt}=2\partial_{z}{\rho}+2\rho\partial_z \rho, \qquad \frac{\hat D\rho}{Dt}=0.$$ Now we search for a solution to this system for which $|\omega^\theta|\approx 1$  and $|\rho|\approx |(\eta,z)|$ near $\eta=z=0$ and we see that the term $\rho\partial_z\rho$ is actually much weaker than $\partial_z\rho$ as $(\eta,z)\rightarrow 0$. This leaves us with the system:
$$ \frac{\hat{D}\omega^\theta}{Dt}=2\partial_{z}{\rho}, \qquad \frac{\hat D\rho}{Dt}=0.$$ This system, finally, has a clear scaling and we search for solutions where $\omega^\theta$ is $0$-homogeneous and $\rho$ is $1$-homogeneous. We then prove that such scale-invariant solutions become singular in finite time. This gives us a clear candidate for how data leading blow-up should look as $(\eta,z)\rightarrow 0$. We then take this data localize it and show that all of the simplifications made above can actually be made rigorous. A heuristic explanation of this is given in Section \ref{sec:heuristic} and the full proof is given in the remaining sections. 

We close this subsection by remarking that this is not the first case when solutions with scale-invariant data were studied in the context of fluid equations. For the $2D$ Euler equations Elling \cite{Elling} recently constructed scale-invariant weak solutions -- though Elling also made use of time-scaling. Scale-invariance has also been used in various ways in the study of the Navier-Stokes system. Leray \cite{Leray34am} conjectured that such solutions could play a key role in the global regularity problem for the Navier-Stokes equation. It was later shown that self-similar blow-up for the Navier-Stokes equation is impossible under some very mild decay conditions in the important works \cite{NRS96} and \cite{TsaiSS98}. Another example is the work of Jia and \v{S}ver\'ak (\cite{JiaSverakSS}, \cite{JiaSverak}) where non-uniqueness of the Leray-Hopf weak solution is established under some spectral assumption on the linearized Navier-Stokes equation around a solution which is initially $-1$ homogeneous in space (see also \cite{TsaiFSS14} and \cite{BradshawTsaiFSS17}).

\subsection{Main Results}
Now we will state the main results. As we have mentioned earlier, our $3D$ domain corresponds to the region $$\Omega_{\epsilon}^{3D}:=\{(x,y,z): (1+\epsilon|z|)^2\leq (x^2+y^2)\} $$ which is obtained from rotating the $2D$ domain \begin{equation*}
\begin{split}
A_\epsilon := \{(r,z):  1+  \epsilon|z| \leq r\}
\end{split}
\end{equation*} with respect to the $z$-axis. Throughout the paper we shall assume that $u^\theta$ and $\omega^\theta$ are respectively even and odd with respect to the plane $\{ z = 0 \}$, so that we may work instead with the $2D$ domain \begin{equation*}
\begin{split}
\Omega_\epsilon := \{(r,z):  0 \le z ,   1+  \epsilon z \leq r\}.
\end{split}
\end{equation*} We remark in advance that the solutions we consider can be taken to vanish smoothly on $\{ z = 0 \}$ so that extending a solution on $\Omega_{\epsilon} $ to $A_\epsilon$ will not affect smoothness of the solution at all. We also define the scale of spaces $\mathring{C}^{0,\alpha}$ introduced in \cite{EJSI} and \cite{E1} using the following norm:
$$\V f\V_{\mathring{C}^{0,\alpha}_{(1,0)}(\Omega_\epsilon)}:= \V f\V_{L^\infty({\Omega_\epsilon})} +\V |\cdot-(1,0)|^\alpha f\V_{C_*^{\alpha}(\Omega_\epsilon)}.$$ Functions belonging to this space are uniformly bounded everywhere and are H\"older continuous away from $(1,0).$ We recall some of the properties of this space in Section \ref{sec:prelim}.  This scale of spaces can be used to propagate boundedness of the vorticity, the full gradient of the velocity field, $\nabla u,$ as well as angular derivatives thereof.

\begin{figure}
	\includegraphics[scale=0.9]{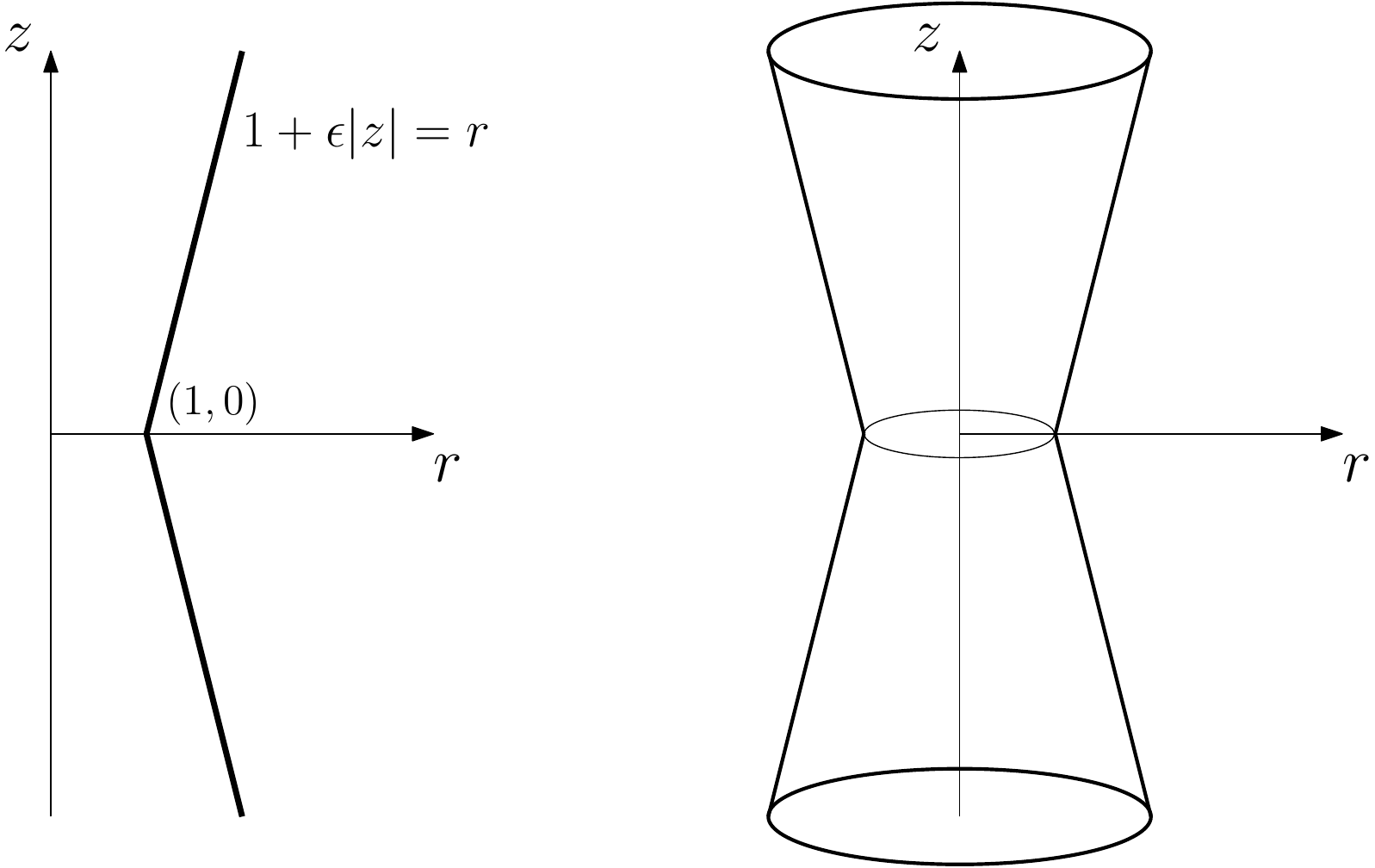}
	\centering
	\caption{Our $2D$ domain $A_\epsilon$ is defined by the region right to the thickened line $\{ 1 + \epsilon z = r \}$ (left). The $3D$ domain is then given by the region exterior to the cylindrical figure (right).}
	\label{fig:domain}
\end{figure}

Our first main result states that the axisymmetric system \eqref{eq:Majda1} -- \eqref{eq:Majda3} is locally well-posed in the scale of spaces $\mathring{C}^{0,\alpha}$.

\begin{customthm}{A}[Local well-posedness]\label{MainThm1}
	Let $\epsilon>0$ and $0<\alpha<1$. 
	For every $\omega_0^\theta$ and $u^\theta_0$ which are compactly supported in $\Omega_\epsilon$ and for which $\omega_0^\theta\in \mathring{C}^{0,\alpha}_{(1,0)}(\Omega_\epsilon)$ and $\nabla u^\theta_0\in \mathring{C}^{0,\alpha}_{(1,0)}(\Omega_\epsilon),$ there corresponds a time $T   >0$ depending only on the $\mathring{C}^{0,\alpha}_{(1,0)}$-norms  and a unique solution pair $(\omega^\theta,u^\theta)$ to the axi-symmetric 3D Euler system \eqref{eq:Majda1} -- \eqref{eq:Majda3} with $\omega^\theta, \nabla u^\theta\in C([0,T); \mathring{C}^{0,\alpha}_{(1,0)}(\Omega_\epsilon))$ and $(\omega^\theta, u^\theta)$ remain compactly supported for all $t\in[0,T)$. The solution can be continued past $T > 0$ if and only if \begin{equation*}
	\begin{split}
	\int_0^T |\omega^\theta(t,\cdot)|_{L^\infty} + |\nabla u^\theta(t,\cdot)|_{L^\infty} dt < + \infty. 
	\end{split}
	\end{equation*}
\end{customthm}

We establish finite time blow-up in this class: 

\begin{customthm}{B}[Finite time singularity formation]\label{MainThm2}
	Let $\epsilon>0$ and $0<\alpha<1$. There exists compactly supported initial data $\omega_0^\theta$ and $u_0^\theta$ for which $\omega_0^\theta\in\mathring{C}^{0,\alpha}_{(1,0)}(\Omega_\epsilon)$ and $\nabla u^\theta_0\in \mathring{C}^{0,\alpha}_{(1,0)}(\Omega_\epsilon)$ whose unique local solution provided by Theorem \ref{MainThm1} blows up at some finite time $T^* > 0$: \begin{equation*}
	\begin{split}
	\lim_{t\rightarrow T^*}\int_0^{t} |\omega^\theta(s,\cdot)|_{L^\infty} + |\nabla u^\theta(s,\cdot)|_{L^\infty} ds =+ \infty. 
	\end{split}
	\end{equation*} The solution may be extended to the domain $A_\epsilon$ with $\omega^\theta, \nabla u^\theta \in  C([0,T^*) ;\mathring{C}^{0,\alpha}_{(1,0)}(A_\epsilon))$. 
\end{customthm}

An immediate corollary is:

\begin{corollary}
	For each $\epsilon>0$, there exists a finite-energy solution $u\in W^{1,\infty}_{x,t}([0,1)\times \Omega^{3D}_\epsilon)$ of the $3D$ incompressible Euler equation with $\lim_{t\rightarrow 1}\int_0^t |\omega(s,\cdot)|_{L^\infty}=+\infty.$
\end{corollary}

On the other hand, as it is expected, the solution is global in time when there is no swirl velocity. 

\begin{customthm}{C}[Global regularity in the no-swirl case]\label{MainThm3}
	Under the assumptions of Theorem \ref{MainThm1}, further suppose that initially $u^\theta_0 \equiv 0$. Then, $u^\theta \equiv 0$ for all time and the solution $\omega^\theta$ exists globally in time. Moreover, $|\omega^\theta(t,\cdot)|_{\mathring{C}^\alpha_{(1,0)}}$ grows at most double exponentially in time: $$ |\omega^\theta(t,\cdot)|_{\mathring{C}_{(1,0)}^\alpha} \le C\exp(C\exp(Ct))$$ with $C > 0$ depending only on $\omega^\theta_0$. 
\end{customthm}

\begin{remark}
	We give a number of remarks regarding the above statements.
	\begin{itemize}
		\item The compact support assumption on initial data are not necessary and can be replaced by some weighted $L^2$-assumption (see Section \ref{sec:elliptic} for details). 
		\item In the statements of Theorem \ref{MainThm1} -- \ref{MainThm2}, the  domain $\Omega_{\epsilon}$ can be replaced by any bounded domain $\Omega$ with smooth boundary except for a point around which it looks like the corner in $\Omega_{\epsilon}$ (see Definition \ref{def:admissible} for the precise requirements).
		\item In the local well-posedness result, the uniqueness statement does not just hold within the class of axi-symmetric solutions but in the class of uniformly Lipschitz and finite energy solutions (to the full $3D$ Euler equations) in the $3D$ domain. 
	\end{itemize}
\end{remark}

\subsection{Disclaimer}

A few months prior to the completion of this work, we posted two articles where we claimed to prove singularity formation for the axi-symmetric 3D Euler equation in the domain $\{(x,y,z): z^2\leq c(x^2+y^2)\}$ for $c$ very small. Unfortunately, those articles contained a major mistake; namely, the system which we were using is not actually the axi-symmetric 3D Euler equation due to a sign error\footnote{Unfortunately this error appears in a few books and papers in mathematical fluid mechanics. We thank Dongyi Wei for pointing this out to us.} in how we wrote the Biot-Savart law. Fortunately, this error does not affect our work on the Boussinesq system nor the present work. 
We should note, however, that the program of using scale-invariant solutions to prove blow-up \emph{is} correct even in that setting; however, it is not clear whether the 1D system associated to that setting has solutions which become singular in finite time.  In the final section of this work, we record the \emph{correct}  $1D$ system for that setting--the blow-up problem for which remains open. Upon inspecting that system, it is clear that there is a mechanism which wants to prevent blow-up. Note that this mechanism is not present here since we are constructing a singularity near $r=1$ and not $r=0$. The domains considered here are also less singular than the ones considered in the previous work and can be taken to be arbitrarily close to a smooth cylinder which is the setting of the numerics of Luo and Hou \cite{HouLuo}. 

\subsection*{Organization of the paper}

In Section \ref{sec:prelim}, we simply recall the definition of scale invariant H\"older spaces $\mathring{C}^{0,\alpha}$ and as well as a few basic properties. Then in Section \ref{sec:heuristic}, we demonstrate heuristically that near the point $(r,z) = (1,0)$, the dynamics of the axisymmetric Euler system for locally scale-invariant data reduces to that for the $2D$ Boussinesq system. The proof of finite time singularity formation for the latter system is reviewed briefly in Section \ref{sec:1d}. In Section \ref{sec:elliptic}, we prove various elliptic estimates which are essential for the proofs of Theorems \ref{MainThm1} and \ref{MainThm3} in Section \ref{sec:lwp} and Theorem \ref{MainThm2} in Section \ref{sec:blowup}. 

\section{Preliminaries}\label{sec:prelim}
In this section, let $D$ be some subset of the plane. Then, the scale-invariant H\"{o}lder spaces are defined as follows:
\begin{definition}
	Let $0 < \alpha \le 1$. Given a function $f \in C^0(D\backslash\{0\})$, we define the $\mathring{C}^{0,\alpha}({D}) = \mathring{C}^{\alpha}({D})$-norm by \begin{equation*}
	\begin{split}
	\V f \V_{\mathring{C}^{\alpha}({D})} &:= \V f \V_{L^\infty({D})} + \V |\cdot|^{\alpha}f \V_{{C}_*^\alpha({D})} \\
	& := \sup_{x \in{ D}} |f(x)| + \sup_{x,x' \in {D}, x \ne x'} \frac{||x|^\alpha f(x)- |x'|^\alpha f(x')|}{|x-x'|^\alpha}.
	\end{split}
	\end{equation*} Then, for $k \ge 1$, we define $\mathring{C}^{k,\alpha}$-norms for $f \in C^k(D\backslash \{0\})$ by \begin{equation}
	\begin{split}
	\V f \V{ \mathring{C}^{k,\alpha}({D})} := \V f \V{ \mathring{C}^{k-1,1}({D})} + \V |\cdot|^{k+\alpha} \nabla^k f \V_{{C}_*^\alpha({D})}.
	\end{split}
	\end{equation} Here, $\nabla^k f$ is a vector consisting of all expressions of the form $\pr_{x_{i_1}} \cdots \pr_{x_{i_k}} f$ where $i_1,\cdots i_k \in \{ 1, 2\}$. Finally, we may define the space $\mathring{C}^\infty$ as the set of functions belonging to all $\mathring{C}^{k,\alpha}$: \begin{equation*}
	\begin{split}
	\mathring{C}^\infty := \cap_{k \ge 0, 0 < \alpha \le 1} \mathring{C}^{k,\alpha}.
	\end{split}
	\end{equation*}
\end{definition}
\begin{remark} From the definition, we note that:
	\begin{itemize}
		\item Let $D = \{ (r,\theta) : r > 0, \theta_1 < \theta < \theta_2 \}$. For a radially homogeneous function $f$ of degree zero, that is, $f(r,\theta) = \tilde{f}(\theta)$ for some function $\tilde{f}$ defined on $[\theta_1,\theta_2]$, we have \begin{equation*}
		\begin{split}
		\V f \V_{ \mathring{C}^{k,\alpha}({D})} = \V \tilde{f} \V_{C^{k,\alpha}[\theta_1,\theta_2]}. 
		\end{split}
		\end{equation*} Similarly, $f \in \mathring{C}^\infty({D})$ if and only if $\tilde{f} \in C^\infty[\theta_1,\theta_2]$.
		\item If $f$ is bounded, then $\V |\cdot|^\alpha f\V_{{C}_*^\alpha} < + \infty$ if and only if (assuming that $|x'| \le |x|$) \begin{equation*}
		\begin{split}
		\sup_{x \ne x', |x-x'|\le c|x|} \frac{ |x|^\alpha  }{|x-x'|^\alpha}|f(x) - f(x')| < + \infty 
		\end{split}
		\end{equation*} for some $c > 0$. 
	\end{itemize}
\end{remark}
\begin{lemma}[Product rule]\label{lem:Holder_product}
	Let $f \in C^\alpha$ with $f(0) = 0$ and $h \in \mathring{C}^\alpha$. Then, we have the following product rule: \begin{equation}\label{eq:Holder_product}
	\begin{split}
	\V fh\V_{C^\alpha} \le C\V h\V_{\mathring{C}^\alpha} \V f \V_{C^\alpha}.
	\end{split}
	\end{equation}
\end{lemma}
\begin{proof}
	Clearly we have that $\V fh\V_{L^\infty} \le \V f \V_{L^\infty} \V h \V_{L^\infty}$. Then, take two points $x \ne x'$ and note that \begin{equation*}
	\begin{split}
	\frac{f(x)h(x) - f(x')h(x')}{|x-x'|^\alpha} = h(x)\frac{f(x) - f(x')}{|x-x'|^\alpha} + \frac{f(x')}{|x'|^\alpha} \cdot \left( \frac{|x|^\alpha h(x) - |x'|^\alpha h(x')}{|x-x'|^\alpha} + \frac{|x'|^\alpha - |x|^\alpha}{|x-x'|^\alpha} h(x) \right)
	\end{split}
	\end{equation*} holds. The desired bound follows immediately. 
\end{proof}
\begin{remark}
	Note that $C^\alpha\not\subset \mathring{C}^\alpha$ since functions belonging to $\mathring{C}^\alpha$ must, in a sense, have decaying derivatives. For example, a function $f\in \mathring{C}^{0,1}$ if and only if it is uniformly bounded and satisfies $|\nabla f(x)|\lesssim \frac{1}{|x|}$ almost everywhere. Of course, any compactly supported $C^\alpha$ function belongs to $\mathring{C}^\alpha$. 
\end{remark}

In the remainder of this paper, we shall take $D$ to be either $\Omega_{\epsilon}$ or an ``admissible'' domain $\Omega$ (see Definition \ref{def:admissible}) in the $(\eta,z)$-coordinates, so that the norm $\mathring{C}^\alpha_{(1,0)}(\Omega_{\epsilon})$ in the $(r,z)$-coordinates used in the statements of the main theorems above is simply the $\mathring{C}^\alpha(\Omega_{\epsilon})$-norm.

\section{A heuristic blow-up proof}\label{sec:heuristic}

The $3D$ axisymmetric Euler equations take the following form in terms of velocities $v  = (v_1, v_2) := (u^r, u^z)$ and $u  :=u^\theta$: 
$$\partial_t v+v\cdot\nabla v+\nabla p=(\frac{u^2}{r},0)$$
$$\div(r v)=0$$
$$\partial_t u+v\cdot\nabla u=-\frac{v_1u}{r}$$
along with $v\cdot n=0$ on the boundary of the spatial domain where $n$ is the exterior unit normal. In the above equations, all derivatives are in $(r,z)$ variables. 
It is easy to see from this formulation that $$\frac{d}{dt} \int_\Omega \left(|v|^2+u^2 \right) rdrdz=0$$ for smooth enough solutions. Moreover, it is also possible to pass to the vorticity formulation by dotting the equation for $v$ with $(\partial_z,-\partial_r).$ Then, for $\omega =\partial_z v_1-\partial_r v_2 $ we have 
$$\partial_t (\frac{\omega}{r})+v\cdot\nabla (\frac{\omega}{r})=\frac{2}{r^2}u\partial_z u,$$
$$\partial_t (r u)+v\cdot\nabla (ru)=0.$$
Since $\div(rv)=0,$ we may write: $$rv=(\partial_z \psi,-\partial_r\psi)$$ with $\psi=0$ on the boundary of the domain (which is consistent with $u\cdot n=0$ on the boundary). Now we can recover $\psi$ from $\omega$ by observing:
$$\frac{1}{r}\partial_{zz}\psi+\partial_r(\frac{1}{r}\partial_r\psi)=\omega.$$
Thus we get the following system:

 \begin{equation}\label{eq:3DEuler}
\begin{split}
&\frac{D}{Dt}\left( \frac{\omega}{r} \right) = \frac{2u\partial_z u}{r^2} ,  \\
&\frac{D}{Dt} \left( r u  \right) = 0. 
\end{split}
\end{equation} Here, \begin{equation}\label{eq:3DEuler2}
\begin{split}
\frac{D}{Dt} := \partial_t + v_1\partial_r + v_2\partial_z
\end{split}
\end{equation} with \begin{equation}\label{eq:3DEuler3}
\begin{split}
v_1 := \frac{\partial_z \psi}{r},\qquad v_2 := -\frac{\partial_r\psi}{r}
\end{split}
\end{equation} and finally, $\psi$ is the solution of \begin{equation}\label{eq:elliptic}
\begin{split}
\tilde{L}\psi = \frac{\omega}{r},\qquad \tilde{L} := \frac{1}{r}\pr_r (\frac{1}{r}\pr_r) + \frac{1}{r^2}\pr_{zz}. 
\end{split}
\end{equation} The system \eqref{eq:3DEuler} -- \eqref{eq:elliptic} is the form of the $3D$ axisymmetric Euler equations that we will use in the remainder of the paper. Now we will show a heuristic blow-up proof before giving the details which can be somewhat technical. 
Let us first set $r:=\eta+1$, $\theta=\arctan(\frac{z}{\eta}),$ $R^2=\eta^2+z^2$. Notice that the domain  $\{1+\epsilon |z|\leq r\}$ becomes $\{\epsilon|z|\leq \eta\}$, which is equal to the sector $$\{(R,\theta): R\geq 0\,\,\, \text{and} \,\,\, \theta\in (-\pi/2+\tan^{-1}\epsilon,\pi/2-\tan^{-1}\epsilon) \}.$$
And the system becomes:
\begin{equation}\label{eq:omega_eta}
\begin{split}
\frac{D}{Dt}\left( \frac{\omega}{\eta+1} \right) = \frac{2u\partial_z u}{(\eta+1)^2} ,
\end{split}
\end{equation}\begin{equation}\label{eq:u_eta}
\begin{split}
\frac{D}{Dt} \left( (\eta+1) u  \right) = 0,
\end{split}
\end{equation} supplemented with \begin{equation}\label{eq:elliptic_eta}
\begin{split}
\tilde{L}\psi = \frac{\omega}{\eta+1},\qquad \tilde{L} := \frac{1}{\eta+1}\pr_\eta (\frac{1}{\eta+1}\pr_\eta) + \frac{1}{(\eta+1)^2}\pr_{zz}.
\end{split}
\end{equation} 

Our goal will be to look for solutions which, near $(\eta,z)=0$, satisfy $\omega\approx g(t,\theta)+\tilde \omega$ and $u\approx 1+RP(t,\theta) +\tilde u$ with $\tilde \omega$ and $\frac{1}{R}\tilde{u}$ vanishing at $0$ like $R^\alpha$.  Let us plug this ansatz into the equation and see what equation $g$ and $P$ must satisfy to ensure the high degree of vanishing of $\tilde \omega$ and $\tilde u$: 

$$\frac{D}{Dt} (\frac{g}{\eta+1})+\frac{D}{Dt}(\frac{\tilde\omega}{\eta+1})=\frac{2}{(1+\eta)^2}(1+RP+\tilde u)\partial_z(RP+\tilde u)$$
$$\partial_t \frac{g}{\eta+1}+v_g\cdot\nabla \frac{g}{\eta+1}+v_{\tilde\omega}\cdot\nabla \frac{g}{\eta+1}+ \frac{D}{Dt} (\frac{\tilde\omega}{\eta+1})=\frac{2}{(1+\eta)^2}\partial_z(RP)+\frac{2}{(1+\eta)^2}\big(\partial_z \tilde u + (RP+\tilde u)\partial_z(RP+\tilde u)\big)$$
Now, notice that the third and fourth terms on the left hand side and the second term on the right hand side all involve quantities which should vanish as $R\rightarrow 0$. Thus, the correct equation for $g$ is:

\begin{equation}\label{g_equation} \partial_t g+\overline{v_g}\cdot\nabla g=2\partial_z(RP)\end{equation} where $\overline{v_g} $ is the highest order term in $v_g$. Indeed, we write: 

$$\tilde{L}\psi_g=g$$ and we believe that $\psi_g=R^2 G(t,\theta)+\tilde \psi$ with $\tilde\psi=o(R^2)$ as $R\rightarrow 0$. Then we observe:
$$\tilde{L}\tilde\psi+\tilde{L}(R^2 G)=g$$
and $\tilde{L}(R^2 G)=\frac{1}{(1+\eta)^2}\partial_{zz}(R^2G)+\frac{1}{1+\eta}\partial_\eta(\frac{1}{1+\eta}\partial_\eta (R^2 G)).$ Now, let us notice that if $\partial_\eta$ hits the $\frac{1}{1+\eta}$ we will get an error term. Thus we see, $$\tilde{L}(R^2G)=4G+G''+O(R)$$ as $R\rightarrow 0$. So we set $4G+G''=g$ and $\tilde\psi=\tilde{L}^{-1}(g-\tilde{L}(R^2)G)=o(R^2).$ Now we see that $v_{g}=\frac{1}{1+\eta}\nabla^\perp(R^2G)+\frac{1}{1+\eta}\nabla^\perp(\tilde\psi)$ and we set $\overline{v_g}=\nabla^\perp(R^2G)$ which is equal to $v_g$ as $R\rightarrow 0$ up to terms which vanish at a controlled rate. This is the justification for \eqref{g_equation}. Next let's write down the equation for $P$. In a similar way to the preceding calculation we see that the correct equation for $P$ is:

\begin{equation}\label{P_equation} \partial_t (RP)+\overline{v_g}\cdot\nabla(\eta+1)+\overline{v_g}\cdot\nabla(RP)=0\end{equation}
Now, using that $\overline{v_g}=\nabla^\perp(R^2G)= 2(z,-\eta)G + (\eta,z)G'$, $\nabla(\eta+1)=(1,0)$, and $\nabla(RP)=\frac{1}{R}(\eta,z)P+\frac{1}{R}(-z,\eta)P',$  we obtain from \eqref{P_equation} after dividing by $R$ that 
$$\partial_t P- 2GP'=-G'P+2\sin(\theta)G+\cos(\theta)G'.$$ Similarly, from \eqref{g_equation}, we get 
$$\partial_t g - 2Gg'= 2\sin(\theta)P+2\cos(\theta)P'.$$
Writing $P=Q+\cos(\theta)$ gives 
$$\partial_t Q-2GQ'=-G'Q$$
$$\partial_t g-2Gg'=2\sin(\theta)Q+2\cos(\theta)Q'.$$ Finally, replacing $g $ and $Q$ with $-g$ and $-Q$ respectively gives $$\partial_t Q+2GQ'=G'Q$$
$$\partial_t g+2Gg'=2\sin(\theta)Q+2\cos(\theta)Q',$$ where $$ G'' + 4G = g.$$
We have already encountered this equation before. It is the same equation as for the Boussinesq system and finite time singularity formation for smooth solutions has already been established.

\begin{remark}
Note that if we had from the beginning realized to write $u=\frac{1}{1+\eta}+RP+\tilde u$ as the correct ansatz, we would not have to have passed from $P$ to $Q$. 
\end{remark}

The above calculation was just a heuristic. In Section \ref{sec:blowup} we will show rigorously that the "remainder" terms which we dropped at each step can actually be dropped by establishing various elliptic estimates (Section \ref{sec:elliptic}) and local well-posedness results (Section \ref{sec:lwp}). 

\section{Blow-up for the $1D$ system}\label{sec:1d}

In this section we recall the results of \cite{EJB} on the analysis of the $1D$ system which was derived (heuristically) above: 
\begin{empheq}[left=\empheqlbrace]{align} 
&\partial_t g+2G\partial_\theta g =2(\sin\theta P+\cos\theta\partial_\theta P),\label{BSI1} \\
\label{BSI2} &\partial_t P+2G\partial_\theta P =P\partial_\theta G,
\end{empheq} on the interval $[0,l]$ where $G$ is obtained from $g$ by solving $\partial_{\theta\theta}G + 4G = g$ subject to the Dirichlet boundary condition $G(0)=G(l)=0$.  In order for the ODE relating $G$ and $g$ to be solvable in general we need to assume $l<\frac{\pi}{2}$. It is then possible to show that there are smooth solutions to this system which become singular in finite time in the sense that there exists $g_0,P_0\in C^\infty([0,l])$ so that the unique local-in-time $C^\infty$  solution pair $(g,P)$ has a maximal forward-in-time interval of existence $[0,T^*)$ and $\lim_{t\rightarrow T^*}|g|_{L^\infty}=+\infty.$ The full details are given in \cite{EJB}. Here we give a sketch of the proof for the convenience of the reader.  

\begin{theorem}
Take $g_0\equiv 0$ and $P_0(\theta)=\theta^2$. Then, for any $l<\frac{\pi}{2}$, the unique local solution to \eqref{BSI1}-\eqref{BSI2} cannot be extended past some $T^*$ and $\lim_{t\rightarrow T^*}|g|_{L^\infty}=+\infty.$
\end{theorem}

\begin{proof}
Assume towards a contradiction that $g$ remains bounded for all time. We will show that $g,g',P,P',P+P''\geq 0$ for all time (here, $'$ refers to $\partial_\theta$). To do this, we first have to establish a fact about the elliptic problem relating $g$ and $G$; namely, if $g\geq 0$ then $G\leq 0$ and consequently $G''\geq g$. This is proven by a maximal principle type argument. It can also be shown that if $g'\geq 0$ also then we have $-G'(0), G'(l), G'(0)+G'(l)\geq 0$ and $G'(l)\geq c\int_0^l g .$
Since $P_0\geq 0,$ inspecting \eqref{BSI2} we see that the solution $P\geq 0$ for all $t>0$. Next, upon differentiating \eqref{BSI2} we see:
$$\partial_t P'+2G\partial_{\theta} P'=-G'P'+PG''.$$
Comparing this with the equation for $g$, \eqref{BSI1}, we see that we can propagate $g\geq 0$ and $P'\geq 0.$ Next, we compute the equations for $g'$ and $P''+P$ and we see:
$$\pr_t g' + 2G\pr_{\theta}g' = - 2G'g' + 2\cos\theta (P + P'')$$
$$\pr_t (P + P'') + 2G\pr_{\theta}(P+P'') = Pg' - 3G'(P + P''). $$
Then it becomes clear that $g'\geq 0$ and $P+P''\geq 0$ can be propagated simultaneously. 
Next, let us compute $\frac{d}{dt} \int_{0}^{l} g(t,\theta)d\theta$:
$$\frac{d}{dt} \int_{0}^{l}   g d\theta =2 \int_{0}^{l}  \partial_\theta G g+4\int_0^l \sin(\theta) P d\theta +2P(l)\cos(l) $$
since $P(0)=0$ for all $t\geq 0$. Now note:
$$\int_0^l g\partial_\theta G d\theta =\frac{1}{2} (G'(l)^2-G'(0)^2)\geq 0.$$
Thus, $$\frac{d}{dt} \int_{0}^{l}  g d\theta \geq cP(l).$$
Now we compute the equation for $P(l)$ and we see:
$$\frac{d}{dt} P(l)=G'(l)P(l)\geq c P(l)\int_{0}^{l}  g d\theta .$$
It then follows that $\int_{0}^{l}  g d\theta $ and $P(l)$ blow up in finite time. This is a contradiction. Thus, the solution could never be global. 
\end{proof}

\begin{remark}
Since our solutions on $[0,l]$ vanish at $\theta=0$, they can be extended by symmetry (odd symmetry for $g$ and even symmetry for $P$), we get smooth solutions on $[-l,l]$ which become singular in finite time. We should note that there do exist simple exact blow-up solutions on $[0,l]$ when $l<\frac{\pi}{2}$ which start out smooth and blow up in finite time; however, these solutions cannot be extended to smooth solutions on $[-l,l]$. 
\end{remark}

\section{Elliptic Estimates}\label{sec:elliptic}
In this section we will establish estimates for the operator $L$ defined by 
\begin{equation}\label{eq:L}
\begin{split}
L(\psi)=\partial_{zz}\psi +\partial_{\eta\eta}\psi - \frac{1}{\eta+1}\partial_{\eta}\psi
\end{split}
\end{equation}
on two types of domains. First, on what we call admissible domains (see Definition \ref{def:admissible}) and then on the domains $\Omega_\epsilon:=\{(\eta,z): 0\leq \epsilon z\leq \eta\}$ for any $\epsilon>0$. Admissible domains are simply bounded domains which look like $\Omega_\epsilon$ near $x:=(\eta,z)=(0,0).$

We briefly recall the main results from \cite{EJB} regarding the Poisson problem on sectors. Given $f \in L^\infty(\Omega_{\epsilon})$, we consider the system \begin{equation}\label{eq:Poisson}
\begin{split}
\Delta \Upsilon  = f, \quad&\mbox{in}\quad \Omega_{\epsilon},\\
\Upsilon = 0,\quad&\mbox{on}\quad \pr\Omega_{\epsilon}.
\end{split}
\end{equation}
\begin{lemma*}[see Lemma 3.2 from \cite{EJB}]
	Given $f \in L^\infty(\Omega_{\epsilon})$, there exists a unique solution to \eqref{eq:Poisson} satisfying $\Upsilon \in W^{2,p}_{loc}$ for all $p < \infty $ and $|\Upsilon(x)| \le C|x|^2$ for some $C > 0$. 
\end{lemma*} From now on, we denote $\Delta^{-1} = \Delta^{-1}_D$ to be the operator $f \mapsto \Upsilon$ for simplicity. Before we proceed, we recall a number of important facts regarding this operator:
\begin{remark}
	Let $f \in L^\infty(\Omega_{\epsilon})$ and $\Upsilon$ be the unique solution provided by the above lemma. \begin{itemize}
		\item The existence statement follows directly from the expression \begin{equation}\label{eq:Poisson_solution}
		\begin{split}
		\Upsilon(x) = \lim_{R \rightarrow +\infty} \int_{  \Omega_{\epsilon} \cap \{|y| < R\} } G_\epsilon(x,y) dy, 
		\end{split}
		\end{equation} where $G_\epsilon$ is the Dirichlet Green's function on $\Omega_\epsilon$ given explicitly by \begin{equation}\label{eq:Green}
		\begin{split}
		G_\epsilon(x,y) = \frac{1}{2\pi} \ln  \frac{|x^{1/\beta}-y^{1/\beta}|}{|\overline{x^{1/\beta}}-y^{1/\beta}|}, \qquad \tan(\beta\pi) = \epsilon^{-1}. 
		\end{split}
		\end{equation} 
		Here we are viewing $x$ and $y$ as complex numbers, e.g. $x = \eta + iz$. The bar denotes the complex conjugate. Note that we may assume $0 < \beta < 1/2$ and $\beta \rightarrow 1/2$ as $\epsilon \rightarrow 0^+$. In the following we shall always assume that $0 < \beta < 1/2$ and it is a function of $\epsilon$ as in \eqref{eq:Green}. 
		\item The kernel for $\nabla\Delta^{-1}$ is given by \begin{equation*}
		\begin{split}
		K_\beta(x,y) := \pr_x G_\beta(x,y) = - \frac{x^{1/\beta -1}}{4\pi\beta} \cdot \frac{\overline{y^{1/\beta}} - y^{1/\beta}}{(x^{1/\beta} - \overline{y^{1/\beta}})(x^{1/\beta} - y^{1/\beta})},
		\end{split}
		\end{equation*} so that $\pr_r\Delta^{-1}f$ and $\pr_z\Delta^{-1}f$ are given by the real and imaginary parts of $K_\beta * f$, respectively. We have the following $L^\infty$ bound: \begin{equation*}
		\begin{split}
		\frac{|\nabla\Delta^{-1}f(x)|}{|x|} \le C| f|_{L^\infty}
		\end{split}
		\end{equation*} for $C = C(\epsilon) > 0$. 
	\end{itemize}
\end{remark}

Given H\"older regularity of $f$ (uniform up to the boundary of $\Omega_\epsilon$), one can show that the second derivatives of $\Delta^{-1}f$ belongs to the same H\"older space.
\begin{lemma}[see Lemmas 3.5 and 3.6, and Corollary 3.8 from \cite{EJB}]\label{lem:Holder_Poisson}
	Given $0 < \epsilon$ and $0 < \alpha <1$, we  have the estimates \begin{equation}\label{eq:log}
	\begin{split}
	| \nabla^2\Delta^{-1}f|_{L^\infty({\Omega}_{\epsilon})} \le C_{\alpha,\epsilon} | f |_{L^\infty({\Omega}_{\epsilon})} \ln\left( 2 + \frac{| f|_{ \mathring{C}^\alpha( {\Omega}_\epsilon)}}{| f |_{L^\infty({\Omega}_{\epsilon})}} \right)
	\end{split}
	\end{equation} and \begin{equation}\label{eq:C^circlealpha}
	\begin{split}
	|\nabla^2\Delta^{-1}f|_{\mathring{C}^\alpha( {\Omega}_\epsilon)} \le C_{\alpha,\epsilon}| f |_{\mathring{C}^\alpha( {\Omega}_\epsilon)}.
	\end{split}
	\end{equation}   for $f $ in $\mathring{C}^\alpha( {\Omega}_\epsilon)$. Moreover, if we have in addition that $\alpha < 1/\beta -2$, where $0 < \beta < 1/2 $ satisfies $\tan(\beta\pi) = \epsilon^{-1}$, we then have \begin{equation}\label{eq:C^alpha_C^circlealpha}
	\begin{split}
	|\nabla^2\Delta^{-1}f|_{\mathring{C}^\alpha \cap C^\alpha( {\Omega}_\epsilon)} \le C_{\alpha,\epsilon}| f |_{\mathring{C}^\alpha \cap C^\alpha( {\Omega}_\epsilon)}
	\end{split}
	\end{equation} for $f \in \mathring{C}^\alpha \cap C^\alpha( {\Omega}_\epsilon)$. 
\end{lemma}

\subsection{Estimates for the Laplacian on Admissible Domains}

 We begin by defining our concept of "admissible domains."
\begin{definition}\label{def:admissible}
A bounded spatial domain $\Omega\subset\mathbb{R}^2$ is said to be \emph{admissible} if
\begin{enumerate} 
\item $\Omega\subset \{\eta\geq 0\}$.
\item $\partial\Omega$ is a simple closed curve
\item $0\in\partial\Omega$ and there exists a $\delta>0$ and a $C^{3}$ diffeomorphism $\Psi:B_\delta(0)\rightarrow B_\delta(0)$ so that $\Psi(0)=0,$ $\nabla\Psi(0)=Id$, and $\partial\Psi(\{(\eta,z):0\leq \epsilon z\leq \eta\}\cap B_\delta(1,0))=\Omega\cap B_\delta(0))$ for some $\epsilon>0$.    
\item $\partial\Omega-B_{\delta}(0)$ is $C^3$.

\end{enumerate}
\end{definition}

\begin{remark}
All our results will work equally well for similar domains which are not simply connected. 
\end{remark}

Let $\Omega$ be an admissible domain in $\mathbb{R}^2.$ Using Grisvard's shift theorem \cite{Gris}, we have that the Dirichlet problem for the Laplacian: $$\Delta\psi=f$$ $$\psi|_{\partial\Omega}=0$$ is uniquely solvable for given $f\in L^p(\Omega)$ and  $D^2\Delta^{-1}_{D}:L^p(\Omega)\rightarrow L^p(\Omega)$ for all $p<\infty$ is a bounded linear operator. We will now show that $D^2\Delta^{-1}_D:\mathring{C}^{0,\alpha}(\Omega)\rightarrow\mathring{C}^{0,\alpha}(\Omega):$

\begin{lemma}\label{lemm:circle_estimate_Laplacian}
Let $\Omega$ be an admissible domain and let $0<\alpha<1.$ Then, there exists a constant $C>0$ so that for all $f\in\mathring{C}^{0,\alpha}(\Omega)$ the unique $W^{2,2}(\Omega)$ solution $\psi$ of the Dirichlet problem:
$$\Delta\psi=f$$
$$\psi|_{\partial\Omega}=0$$ on $\Omega$ satisfies: 
$$|D^2\psi|_{\mathring{C}^{0,\alpha}}\leq C|f|_{\mathring{C}^{0,\alpha}}.$$
\end{lemma} 

\begin{proof}
Let $\epsilon,\delta,$ and $\Psi$ be as in Definition \ref{def:admissible}. Let us first notice that $\psi\in C^{2,\alpha}(\Omega\setminus\{0\})$ using the standard global Schauder estimates \cite{Krylov} since $f\in C^{\alpha}(\Omega\backslash\{0\})\cap L^\infty.$ In the proof we will actually be proving the \emph{a-priori} estimate assuming that $D^2\psi$ belongs to $\mathring{C}^{0,\alpha}(\Omega).$ To show that $D^2\psi$ actually belongs to $\mathring{C}^{0,\alpha}$ all we have to do is exclude a small ball of radius $\epsilon$ around $0$ in all our estimates and then send $\epsilon$ to $0$ by observing that all estimates will be independent of $\epsilon$. We leave that step to the reader. Now we show how to get the \emph{a-priori} estimates. First, let us consider the case where $\Psi(x)=x$ for all $x\in B_\delta(0)$. Let $\phi\in C^{\infty}(\mathbb{R}^2)$ be such that $\phi\equiv 1$ on $B_{\delta/2}(0)$ and $\phi\equiv 0$ on $B_{\delta}(0)^c$. Let $\tilde\psi=\phi\psi$ extended to be identically 0 outside of $B_\delta(0)$. Then, $|\Delta\tilde\psi|_{\mathring{C}^{0,\alpha}(\Omega_\epsilon)}\leq C_{\delta,\epsilon,\alpha} |f|_{\mathring{C}^{0,\alpha}(\Omega)}.$ Moreover, $\tilde\psi=0$ on $\partial\Omega_\epsilon$. Thus, using Lemma \ref{lem:Holder_Poisson},  $$|D^2\psi|_{\mathring{C}^{0,\alpha}(B_{\delta/2}(0))}\leq |D^2\tilde\psi|_{\mathring{C}^{0,\alpha}(\Omega_\epsilon)}\lesssim |f|_{\mathring{C}^{0,\alpha}(\Omega)}.$$ 
This establishes the estimates near the corner. 
Now we notice: $|f|_{C^\alpha(\Omega\setminus B_{\frac{\delta}{8}}(0))}\lesssim |f|_{\mathring{C}^{0,\alpha}(\Omega)}.$ Now we use the global Schauder estimates (see \cite{Krylov}) on $\Omega\setminus B_{\frac{\delta}{8}}(0)$ using the equation for $\psi$ to deduce: $$|D^2\psi|_{C^\alpha(\Omega\setminus B_{\frac{\delta}{4}}(0))}\lesssim |f|_{\mathring{C}^{0,\alpha}(\Omega)}.$$ 
This establishes the theorem in the case where $\Psi(x)=x$ for all $x\in B_\delta(0)$. Notice that the proof above just consisted of two cases: the region near $0$ and the region away from $0$. The estimate for the region away from $0$ will not change when $\Psi$ is variable coefficient. However, near zero we will just have to use the usual method of freezing the coefficients. Now, since the Laplacian commutes with rotations, we might as well assume that $\Psi(x)=x+\Phi(x)$ with $\Phi\in C^3$ and $\Phi(0)=\nabla\Phi(0)=0$.  Next we will let $\tilde\psi=\phi(\psi\circ\Psi)$ extended to be $0$ outside of $B_\delta(0).$ Note that $\tilde\psi$ vanishes on $\partial\Omega_\epsilon$. By studying $\Delta\tilde\psi$ we see that 
if $\zeta<\delta$ we have $$|\Delta\tilde\psi|_{\mathring{C}^{0,\alpha}(B_\zeta)}\leq |f|_{\mathring{C}^{0,\alpha}(\Omega)}+C_{\delta}\zeta|D^2\tilde\psi|_{\mathring{C}^{0,\alpha}(B_\zeta)}+C_\delta|\nabla\tilde\psi|_{\mathring{C}^{0,\alpha}(B_\zeta)}.$$ Since $\tilde\psi$ vanishes on $\partial\Omega_\epsilon$, we have that $|D^2\tilde\psi|_{\mathring{C}^{0,\alpha}(\Omega_\epsilon)}\leq C_{\epsilon} |\Delta\tilde\psi|_{\mathring{C}^{0,\alpha}(\Omega_\epsilon)}.$ Notice also that since $\tilde\psi$ vanishes on $\partial{\Omega_\epsilon}$ we must have $|\nabla\tilde\psi|_{\mathring{C}^{0,\alpha}(B_\zeta)}\leq C\zeta|D^2\tilde\psi|_{\mathring{C}^{0,\alpha}(B_\zeta)}$. 
Thus, taking $\zeta$ small enough (depending only on $\epsilon, \delta, $ and $\alpha$), we have: 
$$|\Delta\tilde\psi|_{\mathring{C}^{0,\alpha}(B_\zeta)}\leq |f|_{\mathring{C}^{0,\alpha}(\Omega)}+C\zeta|D^2\tilde\psi|_{\mathring{C}^{0,\alpha}(B_\zeta^c)}.$$
But we already know that $|D^2\tilde\psi|_{\mathring{C}^{0,\alpha}(B_\zeta^c)}\leq C |f|_{\mathring{C}^{0,\alpha}(\Omega)}.$ Thus we get: $$|\Delta\tilde\psi|_{\mathring{C}^{0,\alpha}(B_\zeta)}\leq C|f|_{\mathring{C}^{0,\alpha}(\Omega)}.$$
This finishes the proof of the estimate $$|D^2\psi|_{\mathring{C}^{0,\alpha}(\Omega)}\leq C|f|_{\mathring{C}^{0,\alpha}(\Omega)}.$$ 
\end{proof}

The exact same proof yields classical $C^\alpha$ estimates on admissible domains when  $\alpha < 1/\beta -2$ (recall that $\tan(\beta\pi) = \epsilon^{-1}$) using the corresponding estimates for the Laplacian on the sectors $\Omega_\epsilon$ and freezing the coefficients as above.

\begin{lemma}\label{lemm:schauder_estimate}
Let $\Omega$ be an admissible domain with $\epsilon$ as in Definition \ref{def:admissible}. Let  $\alpha < 1/\beta -2$. Then, there exists a constant $C>0$ so that for all $f\in{C}^{0,\alpha}(\Omega)$ the unique $W^{2,2}(\Omega)$ solution $\psi$ of the Dirichlet problem:
$$\Delta\psi=f$$
$$\psi|_{\partial\Omega}=0$$ on $\Omega$ satisfies: 
$$|D^2\psi|_{{C}^{0,\alpha}}\leq C|f|_{{C}^{0,\alpha}}.$$
\end{lemma} 

\subsection{Estimates for $L$ on admissible domains}
Now we move to establish estimates on the axi-symmetric Biot-Savart operator. Recall that the operator $L$ was defined by by $$L(\psi)=\Delta\psi -\frac{1}{r+1}\partial_r\psi.$$ 
\begin{lemma}\label{lemm:circle_estimate_L}
Let $\Omega$ be an admissible domain and let $0<\alpha<1.$ Then, there exists a constant $C>0$ so that for all $f\in\mathring{C}^{0,\alpha}(\Omega)$ the unique $W^{2,2}(\Omega)$ solution $\psi$ of the Dirichlet problem:
$$L(\psi)=f$$
$$\psi|_{\partial\Omega}=0$$ on $\Omega$ satisfies: 
$$|D^2\psi|_{\mathring{C}^{0,\alpha}}\leq C|f|_{\mathring{C}^{0,\alpha}}.$$
\end{lemma}

\begin{proof}
For existence of a $W^{2,2}$ solution we are relying on Grisvard's shift theorem \cite{Gris}. However one could avoid using the shift theorem by using the \emph{a-priori} estimates we will now prove along with the continuity method (for more details see the proof of Lemma \ref{lem:apriori} in the next subsection). Using the standard Schauder theory, we have, for any $\zeta>0$, 
$$|D^2\psi|_{\mathring{C}^{0,\alpha}(B_{\zeta}^c)}\leq C |f|_{\mathring{C}^{0,\alpha}(\Omega)}$$ for some $C>0$ depending on $\zeta$ and $\alpha$. In fact, in this estimate, we could have the $C^{\alpha}$ norm on the left side of the inequality.  Notice, however, that $|\partial_r\psi|_{\mathring{C}^{0,\alpha}(B_{2\zeta})}\leq 10\zeta |D^2\psi|_{\mathring{C}^{0,\alpha}(B_{2\zeta})}$ since $\psi=0$ on $\partial\Omega$. Thus, as before, \begin{equation*}
\begin{split}
 |D^2\psi|_{\mathring{C}^{0,\alpha}(\Omega)}&\leq C_{\Omega,\alpha}|\Delta\psi|_{\mathring{C}^{0,\alpha}(\Omega)}\leq C_{\Omega,\alpha}|f|_{\mathring{C}^{0,\alpha}(\Omega)}+ C_{\Omega,\alpha}\left|\frac{1}{1+r}\partial_r\psi \right|_{\mathring{C}^{0,\alpha}(\Omega)}\\
& \leq C_{\Omega,\alpha,\zeta}|f|_{\mathring{C}^{0,\alpha}(\Omega)}+C_{\Omega,\alpha}\zeta|D^2\psi|_{\mathring{C}^{0,\alpha}(\Omega)},
\end{split}
\end{equation*} where the first inequality uses Lemma \ref{lemm:circle_estimate_Laplacian}, the second inequality uses $L(\psi)=f$, the third inequality uses that $\Omega\subset \{r\geq 0\}$, and the last inequality uses the estimate for $D^2\psi$ on $B_\zeta^c$ above. Notice that the first constant in the last inequality may depend non-trivially on $\zeta$ (in fact, it will become unbounded as $\zeta\rightarrow 0$) while we make the dependence on $\zeta$ explicit in the second constant. Now we take $\zeta$ small depending on $C_{\Omega,\alpha}$ from the last inequality and we are done. 
\end{proof}

Similarly, we have the full Schauder estimates for $L^{-1}$ which is the content of the following lemma.

\begin{lemma}\label{lemm:schauder_estimate_L}
Let $\Omega$ be an admissible domain with $\epsilon$ as in Definition \ref{def:admissible}. Let $0<\alpha<\epsilon.$ Then, there exists a constant $C>0$ so that for all $f\in{C}^{0,\alpha}(\Omega)$ the unique $W^{2,2}(\Omega)$ solution $\psi$ of the Dirichlet problem:
$$L(\psi)=f$$
$$\psi|_{\partial\Omega}=0$$ on $\Omega$ satisfies: 
$$|D^2\psi|_{{C}^{0,\alpha}}\leq C|f|_{{C}^{0,\alpha}}.$$
\end{lemma} 
We leave the details to the reader. Next, we observe the following simple corollary which is of great importance. 
\begin{corollary}[$L\psi$ vanishing to order $\alpha$ implies $\psi$ vanishes to order $2+\alpha$]\label{cor:vanishing}
Let $\Omega$ be an admissible domain with $\epsilon$ as in Definition \ref{def:admissible} and assume that $f \in C^{\alpha}(\Omega)$ with $\alpha<\epsilon$. Assume $\psi$ is the unique solution of $L\psi=f $ with $\psi=0$ on $\partial\Omega$ from Lemma \ref{lemm:schauder_estimate_L}. Then, if $f(0)=0$, $D^2\psi(0)=0$. In particular, $|\psi(x)|\leq C_{\alpha,\Omega}|f|_{C^\alpha} |x|^{2+\alpha}$ for $x\in \Omega$. 
\end{corollary}

\begin{proof}
Assume $\Omega$ is an admissible domain and let $\epsilon,\delta,\Psi$ be as in Definition \ref{def:admissible}. Let $\phi\in C^{\infty}(\mathbb{R}^2)$ be such that $\phi\equiv 1$ on $B_{\delta/2}(0)$ and $\phi\equiv 0$ on $B_{\delta}(0)^c$. By Lemma \ref{lemm:schauder_estimate_L}, $\psi\in C^{2,\alpha}(\Omega)$. Define $\tilde\psi=\phi (\psi\circ\Psi)$ extended to be 0 outside of $B_{\delta}(0)$. Let's notice that inside of $B_{\frac{\delta}{2}}(0)$, $L(\tilde\psi\circ\Psi^{-1})=f$ and that $\tilde\psi$ vanishes along $z=0$ and $\epsilon z=r$. This already implies that $\Psi$ vanishes quadratically and that $\partial_{rr}\psi$ and $(\partial_{z}+\epsilon\partial_r)^2\psi$ vanish at $0$. Now we will use that $L\psi$ vanishes at $0$ to conclude that $\partial_{rz}\psi$ and $\partial_{zz}\psi$ both vanish at $0$ which then will conclude the proof. Notice that $\nabla\tilde\psi(0)=0.$ 
$$f= L(\tilde\psi\circ\Psi^{-1})=\Delta (\tilde\psi\circ\Psi^{-1})+\frac{1}{r+1}\partial_r(\tilde\psi\circ\Psi^{-1}).$$
Thus, evaluating at $0$ and using that $\nabla\tilde\psi(0)=0$ and $\nabla\Psi^{-1}(0)=Id$ we see:
$$0=\div(\nabla\Psi^{-1}\nabla\tilde\psi\circ\Psi^{-1})|_{x=0}=\Delta\tilde\psi(0).$$ Then, using that $\partial_{rr}\tilde\psi(0)=(\partial_{z}+\epsilon\partial_r)^2\psi(0)=0$ we get that $\partial_{rr}\tilde\psi(0)=\partial_{rz}\tilde\psi(0)=\partial_{zz}\tilde\psi(0)=0$. Thus, $D^2\psi(0)=D^2\tilde\psi(0)=0.$ Then, since $\psi\in C^{2,\alpha}$, $|\psi(x)|\lesssim |f|_{C^\alpha}|x|^{2+\alpha}$ and we are done. 
\end{proof}
\begin{remark}
The above proof breaks down when $\epsilon=0$ since one could only say that $\partial_{rr}\psi(0)=\partial_{zz}\psi(0)=0$. 
\end{remark}

\subsection{Estimates on $\Omega_\epsilon$}

Estimates for $L$ on $\Omega_\epsilon$ are slightly more cumbersome than on admissible domains since $\Omega_\epsilon$ is unbounded. In fact, while we showed that $D^{2}\Delta^{-1}:\mathring{C}^{0,\alpha}(\Omega_\epsilon)\rightarrow\mathring{C}^{0,\alpha}(\Omega_\epsilon)$ for all $\epsilon>0$, it may not be true that $D^2L^{-1}$ satisfies the same property. We shall impose some mild $L^2$-type decay assumption on the vorticity to achieve this. 

Fix some $\epsilon > 0$ and $0 < \alpha <1$. In the remainder of this section, we shall suppress from writing out the dependence of multiplicative constants on $\epsilon$ and $\alpha$. Let us define the space $\mathcal{B}$ by the collection of functions $\psi$ defined in $\Omega_{\epsilon}$, twice differentiable in $\Omega_{\epsilon}\backslash\{0\}$, satisfying \begin{itemize}
\item $\psi = 0$ on $\partial\Omega_{\epsilon}$, 
\item $\nabla^2\psi \in \mathring{C}^{0,\alpha}(\Omega_{\epsilon})$,
\item $(1 + |x|)\Delta \psi(x) \in L^2(\Omega_{\epsilon}), $
\item $ \nabla\psi(x) \in L^2(\Omega_{\epsilon})$,
\item $ (1 + |x|)^{-1} \psi(x) \in L^2(\Omega_{\epsilon})$. 
\end{itemize}  We simply define the norm on $\mathcal{B}$ by \begin{equation}\label{eq:B_norm}
\begin{split}
| \psi |_{\mathcal{B}} := | \nabla^2\psi |_{ \mathring{C}^{0,\alpha}(\Omega_{\epsilon}) } + | (1 + |x|) \Delta \psi(x) |_{ L^2(\Omega_{\epsilon}) } + | \nabla \psi(x) |_{ L^2(\Omega_{\epsilon}) } + | (1 + |x|)^{-1} \psi(x) |_{ L^2(\Omega_{\epsilon}) }. 
\end{split}
\end{equation} On the other hand, we define $\mathcal{V}$ be the space of bounded functions $f$ in $\Omega_{\epsilon}$ satisfying \begin{itemize}
\item $f \in  \mathring{C}^{0,\alpha}(\Omega_{\epsilon})$,
\item $(1 + |x|)|f(x)| \in L^2(\Omega_{\epsilon})$. 
\end{itemize} Then, we set \begin{equation}\label{eq:V_norm}
\begin{split}
| f |_{\mathcal{V}} := | f |_{\mathring{C}^{0,\alpha}(\Omega_{\epsilon})} + | (1 + |x|)f(x) |_{ L^2(\Omega_{\epsilon}) }.
\end{split}
\end{equation} Note that the spaces $\mathcal{B}$ and $\mathcal{V}$ are Banach spaces. This is clear for $\mathcal{V}$, and to see this for $\mathcal{B}$, let $\{ \psi_n \}_{n \ge 1}$ be a Cauchy sequence with respect to the $\mathcal{B}$-norm. Then, for some function $g \in \mathcal{V}$, we have convergence $\Delta \psi_n \rightarrow g $ in the norm $\mathring{C}^{0,\alpha}(\Omega_{\epsilon})$. At this point, we know that there exists a unique $\psi $ which satisfies the Dirichlet boundary condition, $\Delta\psi = g$, and $\nabla^2\psi \in \mathring{C}^{0,\alpha}(\Omega_{\epsilon})$. It only remains to show the $L^2$ bounds for $\nabla\psi$ and $(1 + |x|)^{-1}\psi$, and this part is included in the proof of Lemma \ref{lem:apriori} below. 

We consider for $t \in [0,1]$ the family of operators \begin{equation*}
\begin{split}
L_t = \Delta - \frac{t}{1+r} \pr_r ,
\end{split}
\end{equation*} so that $L_0 = \Delta$ and $L_1 = L$. It is clear that $L_t$ defines a bounded linear operator from $\mathcal{B}$ to $\mathcal{V}$. In the lemma below, we shall obtain the following estimate \begin{equation}\label{eq:apriori_uniform}
\begin{split}
|\psi |_{\mathcal{B}}  \le C |L_t \psi|_{\mathcal{V}}
\end{split}
\end{equation} where the constant $C > 0$ is independent of $t \in [0,1]$. 

\begin{lemma}\label{lem:apriori}
Fix some $0 < \alpha <1$ and $0 < \epsilon$. For any $f \in \mathcal{V}$, there exists a unique solution $\psi \in \mathcal{B} $ for $L\psi = f$ satisfying \begin{equation}\label{eq:apriori}
\begin{split}
|\psi |_{\mathcal{B}}  \le C |f|_{\mathcal{V}}
\end{split}
\end{equation} with some constant $ C > 0$. 
\end{lemma}
 
\begin{proof}

We begin by noting that the uniqueness and existence is immediate once we prove the uniform estimate \eqref{eq:apriori_uniform}. Indeed, this estimate guarantees invertibility of $L_t$ for all $t \in [0,1]$ by the method of continuity (see \cite[Theorem 5.2]{GT}), once we prove the invertibility in the case $t = 0$ (Case (i) below). We let $f \in \mathcal{V}$ and $\psi \in \mathcal{B}$ satisfy \begin{equation}\label{eq:psif}
\begin{split}
L_t\psi =  \Delta \psi - t \frac{1}{1+r} \pr_r \psi = f,
\end{split}
\end{equation} and first deal with the Laplacian case. 

\medskip

\textbf{(i) Case of the Laplacian $t = 0$.}

\medskip

In this special case $t = 0$, to obtain the bound \eqref{eq:apriori_uniform}, it suffices to prove that \begin{equation}\label{eq:Laplacian_bound}
\begin{split}
|\nabla\psi|_{L^2} + |(1 + |x|)^{-1}\psi|_{L^2} \le C |(1 + |x|)\Delta\psi|_{L^2}.
\end{split}
\end{equation} We first note that \begin{equation*}
\begin{split}
\left|\int_{\Omega_{\epsilon}} \psi \Delta\psi \right| \le |(1 + |x|)^{-1}\psi|_{L^2} | (1 + |x|)\Delta\psi|_{L^2}
\end{split}
\end{equation*} so that the integral on the left hand side is well-defined for $\psi \in \mathcal{B}$. Next, we write \begin{equation*}
\begin{split}
\int_{\Omega_{\epsilon}} \psi \Delta\psi &= \lim_{R \rightarrow +\infty} \int_{\Omega_{\epsilon} \cap B_0(R)}  \psi\Delta\psi \\
& = -\lim_{R \rightarrow +\infty} \int_{\Omega_{\epsilon} \cap B_0(R)} |\nabla\psi|^2 + \lim_{R \rightarrow +\infty} \int_{\partial\left(\Omega_{\epsilon} \cap B_0(R)\right)} \psi \pr_n\psi . 
\end{split}
\end{equation*} Using the boundary condition for $\psi$, the last integral term reduces to \begin{equation*}
\begin{split}
\lim_{R \rightarrow +\infty} \int_{\Omega_{\epsilon} \cap\partial\left( B_0(R)\right)} \psi \pr_n\psi 
\end{split}
\end{equation*} and then using the $L^2$-bounds for $(1 + |x|)^{-1} \psi$ and $\nabla\psi$ it is possible to extract a sequence $R_n \rightarrow +\infty$ such that the above boundary integral decays to zero in absolute value. This justifies the integration by parts formula \begin{equation*}
\begin{split}
\int_{\Omega_{\epsilon}} \psi \Delta\psi = - \int_{\Omega_{\epsilon} } |\nabla\psi|^2
\end{split}
\end{equation*} for $\psi \in \mathcal{B}$. In particular, we obtain \begin{equation}\label{eq:Laplacian_L2_1}
\begin{split}
|\nabla\psi|_{L^2}^2 \le |(1 + |x|)^{-1}\psi|_{L^2} |(1 + |x|)\Delta\psi|_{L^2} . 
\end{split}
\end{equation} Next, we compute \begin{equation*}
\begin{split}
 |(1 +r)^{-1}\psi|_{L^2}^2 & = \int_{\Omega_{\epsilon}} (1 +r)^{-2}   |\psi(x)|^2   = \int_{\Omega_{\epsilon}} -\pr_r (1 + r)^{-1} |\psi(x)|^2 \\
 & = \int_{\Omega_{\epsilon}} (1 + r)^{-1} 2\psi \pr_r \psi \le 2 |(1 + r)^{-1} \psi|_{L^2} | \nabla \psi |_{L^2}  
\end{split}
\end{equation*} (integration by parts can be justified similarly as above) so that \begin{equation}\label{eq:Lap_L2_2}
\begin{split}
|(1 +|x|)^{-1}\psi|_{L^2}\le c|(1 +r)^{-1}\psi|_{L^2}\le   C| \nabla\psi|_{L^2}.
\end{split}
\end{equation} In the above we have used that $c'|x| \le r \le |x| $ on $\Omega_{\epsilon}$. Estimates \eqref{eq:Laplacian_L2_1} and \eqref{eq:Lap_L2_2} imply \begin{equation*}
\begin{split}
c|(1 +|x|)^{-1}\psi|_{L^2} \le |\nabla\psi|_{L^2}\le C  |(1 + |x|)\Delta\psi|_{L^2}.
\end{split}
\end{equation*} This finishes the proof of \eqref{eq:apriori_uniform} in the special case $t = 0$. 

\medskip

\textbf{(ii) General case.}

\medskip

We now treat the case $t > 0$. We proceed in a number of steps.

\medskip

\emph{Step 1: $H^1$-estimates}

\medskip

Multiplying both sides of \eqref{eq:psif} by $\psi$ and integrating we see:
$$-\int_{\Omega_\epsilon} |\nabla\psi|^2 -\frac{t}{2}\int_{\Omega_\epsilon} \frac{1}{(r+1)^2}\psi^2= \int_{\Omega_\epsilon} f\psi.$$ Notice again that $L^2$-assumptions on $\psi$ and $\nabla\psi$ in the definition of $\mathcal{B}$ justify applying integration by parts. 
Now using the Cauchy-Schwarz inequality gives 
$$ |\nabla\psi|_{L^2}^2  \leq  |f|_{\mathcal{V}} |(1 + |x|)^{-1}\psi|_{L^2} .$$
Recalling \eqref{eq:Lap_L2_2} gives \begin{equation*}
\begin{split}
c|(1 + |x|)^{-1}\psi|_{L^2}  \le |\nabla\psi|_{L^2} \le C|f|_{\mathcal{V}}
\end{split}
\end{equation*} and then using the equation gives \begin{equation*}
\begin{split}
|(1+|x|)\Delta\psi|_{L^2} \le C|\nabla\psi|_{L^2} + C|(1 + |x|)f |_{L^2} \le 2C|(1 + |x|)f|_{L^2}. 
\end{split}
\end{equation*}

\medskip

\emph{Step 2: $H^2$ estimates}

\medskip

We now use a well known inequality which holds in convex domains (see \cite{TaylorI} and \cite{Gris}).  For simplicity, we will only give the calculation in the case $\epsilon=1$. Notice that:
$$0=\int_{\partial\Omega_1}(\partial_{r}\psi-\partial_z\psi)\partial_r\nabla^\perp\psi\cdot n,$$ where $n$ is the unit exterior normal to $\partial\Omega_\epsilon$. This follows from the vanishing of $\psi$ on $\partial\Omega_\epsilon$. This implies $$0=\int_{ \Omega_1}\div\Big((\partial_{r}\psi-\partial_z\psi)\partial_r\nabla^\perp\psi\Big)=-\int_{ \Omega_1}\nabla\partial_z\psi\cdot\nabla^\perp\partial_r\psi=\int_{\Omega_{1}}\partial_{zz}\psi\partial_{rr}\psi- (\partial_{rz}\psi)^2.$$ In particular, $$\int_{\Omega_1}\partial_{zz}\psi\partial_{rr}\psi=\int_{\Omega_1}(\partial_{rz}\psi)^2.$$  Then, $|D^2\psi|_{L^2}\leq|\Delta\psi|_{L^2}\leq C|(1 + |x|)f|_{L^2}.$ (The assumption $\psi \in \mathcal{B}$ was again used to justify convergence of the integrals as well as integration by parts.) Note that the estimates given above only improve in convex domains where the boundary has non-zero curvature (see \cite{TaylorI} and \cite{Gris}). This concludes the $H^2$ estimates. 

\medskip

\emph{Step 3: $\frac{1}{1+r}\partial\psi\in L^4$}

\medskip

We now move to prove higher integrability of $\partial_r\psi$. While we could use the Sobolev embedding theorem for domains with corners as in Grisvard \cite{Gris}, we wish to keep this work as self-contained as possible. Observe that 
$$\int_{\Omega_\epsilon}\frac{1}{(1+r)^4}(\partial_r\psi)^4=\int_{\Omega_\epsilon}\div( (\partial_r\psi)^3 (\psi,0)\frac{1}{(1+r)^4})-3\int_{\Omega_\epsilon} (\partial_{r}\psi)^2\partial_{rr}\psi\psi+4\int_{\Omega_\epsilon} (\partial_{r}\psi)^3\psi \frac{1}{(1+r)^5}$$
$$=-3\int_{\Omega_\epsilon} \frac{(\partial_{r}\psi)^2}{(1+r)^4}\partial_{rr}\psi\psi+4\int_{\Omega_\epsilon} (\partial_{r}\psi)^3\psi \frac{1}{(1+r)^5},$$ since $\psi=0$ on $\partial\Omega_\epsilon$. Now using the Cauchy-Schwarz inequality we get:
$$\int_{\Omega_\epsilon} \frac{(\partial_{r}\psi)^4}{(1+r)^4}\leq C\int_{\Omega_\epsilon} (\partial_{rr}\psi)^2\frac{\psi^2}{(1+r)^4}.$$ 
Next we see: $$|\frac{\psi}{(r+1)^2}|_{L^\infty}\leq \int_{\Omega_\epsilon} |\partial_{rz}\frac{\psi}{(r+1)^2}| \leq |\nabla\psi|_{H^1}\leq C|(1 + |x|)f|_{L^2},$$ where the first inequality follows from writing: $$f(r,z)=\int_0^z \partial_{2}f(r,w)dw=-\int_{r}^\infty\int_{0}^z\partial_{12}f(u,w)dwdu$$ for $f\rightarrow 0$ at $\infty$. This is reminiscent of the well known embedding of $W^{2,1}(\mathbb{R}^2)$ into $L^\infty.$
Now we see: $$|\frac{1}{1+r}\partial_r\psi|_{L^4}\leq C|(1 + |x|)f|_{L^2}.$$

\medskip

\emph{Step 4: $\Delta_{D}^{-1}: L^2\cap L^4\rightarrow \dot{W}^{1,\infty}$}

\medskip

Next we will show that solutions to $\Delta\psi=g$ with $g\in L^2\cap L^4$ must satisfy $\frac{1}{1+r}\partial_{r}\psi\in L^\infty.$ Recall from Section 3.1 of \cite{EJB} that $$\nabla\psi(x)= \int_{\Omega_\epsilon} K(x,y)g(y)dy,$$ where $K$ satisfies $$|K(x,y)|\leq \frac{C}{|x-y| }$$ as well as $$ |K(x,y)| \le \frac{C|x|}{|y|^2} $$ in the region $|y| \gtrsim |x|$. 
In particular, estimating separately the regions $|x-y| \le |x|/2$ and $|x-y| > |x|/2$ we get  $$|\frac{1}{1+r}\nabla\psi|_{L^\infty}\leq |g|_{L^2\cap L^4}.$$ Now applying this to our situation, we get from $$\Delta \psi=f+\frac{t}{r+1}\pr_r\psi$$ that $$|\frac{1}{1+r} \nabla \psi |_{L^\infty}\lesssim |f|_{L^2 \cap L^4} + |\frac{1}{1+r} \nabla \psi|_{L^2 \cap L^4} \lesssim |f|_{L^\infty} + |(1 + |x|)f|_{L^2} .$$ Now we study solutions of the Dirichlet problem with bounded right-hand-side. 

\medskip

\emph{Step 5: $\frac{1}{1+r}\nabla\Delta_{D}^{-1}:L^\infty\rightarrow \mathring{C}^{0,\alpha}$}

\medskip

First, we know from \cite{EJB} (Lemma 3.2 there and its proof) that $$\left|\frac{1}{1+r}\nabla\Delta^{-1}_{D}g\right|_{L^\infty}\lesssim \left|\frac{1}{r}\nabla\Delta^{-1}_{D} g\right|_{L^\infty}\lesssim |g|_{L^\infty}.$$ Next we prove the $\mathring{C}^{0,\alpha}$ estimate. It suffices to show that for $|a_1-a_2|<1$ and $|a_2| \le |a_1|$, we have:
$$\frac{1 + |a_1|^\alpha}{1+|a_1|}\int_{\Omega_\epsilon} (K(a_1,b)-K(a_2,b))g(b)db\lesssim |a_1-a_2|^{\alpha}$$ for any $\alpha<1$ where (recall that $\tan(\beta\pi) = \epsilon^{-1}$)  $$K(a,b)=-\frac{a^{\frac{1}{\beta}-1}}{4\pi\beta}\frac{\overline{b^{1/\beta}}-b^{1/\beta}}{(a^{1/\beta}-\overline {b^{{1/\beta}}})(a^{1/\beta}-b^{1/\beta})}.$$
Now we see: $$|K(a_1,b)-K(a_2,b)|\lesssim |\bar{b}^{\frac{1}{\beta}}-b^\frac{1}{\beta}|| \frac{a_1^{\frac{1}{\beta}-1}}{(a_1^\frac{1}{\beta}-\overline{b^{\frac{1}{\beta}}})(a_1^{\frac{1}{\beta}}-b^{\frac{1}{\beta}})}-\frac{a_2^{\frac{1}{\beta}-1}}{({a_2}^\frac{1}{\beta}-\overline{b^{\frac{1}{\beta}}})(a_2^{\frac{1}{\beta}}-b^{\frac{1}{\beta}})}|$$
$$\lesssim \frac{|a_1|^{\frac{1}{\beta}-2}|a_1-a_2||\bar{b}^{\frac{1}{\beta}}-b^\frac{1}{\beta}|}{|a_1^\frac{1}{\beta}-\overline{b}^{\frac{1}{\beta}}||a_1^{\frac{1}{\beta}}-b^{\frac{1}{\beta}}|}+\frac{|a_1|^{\frac{1}{\beta}-1}| (a_1^\frac{1}{\beta}-\bar{b}^\frac{1}{\beta})(a_1^\frac{1}{\beta}-b^\frac{1}{\beta})-(a_2^\frac{1}{\beta}-\bar{b}^\frac{1}{\beta})(a_2^\frac{1}{\beta}-b^\frac{1}{\beta})||b^\frac{1}{\beta}-\bar{b}^\frac{1}{\beta}|}{|a_1^\frac{1}{\beta}-\bar{b}^\frac{1}{\beta}||a_1^\frac{1}{\beta}-{b}^\frac{1}{\beta}||a_2^\frac{1}{\beta}-\bar{b}^\frac{1}{\beta}||a_2^\frac{1}{\beta}-{b}^\frac{1}{\beta}|}.$$
The first term is easy to deal with while, for the second term, we use: \begin{equation*}
\begin{split}
&|(a_1^\frac{1}{\beta}-\bar{b}^\frac{1}{\beta})(a_1^\frac{1}{\beta}-b^\frac{1}{\beta})-(a_2^\frac{1}{\beta}-\bar{b}^\frac{1}{\beta})(a_2^\frac{1}{\beta}-b^\frac{1}{\beta})|=|a_1^{1/\beta}-a_2^{1/\beta}||a_1^{1/\beta}+a_2^{1/\beta}-\bar{b}^{1/\beta}-b^{1/\beta}|\\
&\quad \leq |a_1^{1/\beta}-a_2^{1/\beta}|(|a_1^{1/\beta}-\bar{b}^{1/\beta}|+|b^{\frac{1}{\beta}}-a_2^{1/\beta}|)= |a_1^{1/\beta}-a_2^{1/\beta}|^{\alpha}|a_1^{1/\beta}-a_2^{1/\beta}|^{1-\alpha}(|a_1^{1/\beta}-\bar{b}^{1/\beta}|+|b^{\frac{1}{\beta}}-a_2^{1/\beta}|)\\
&\quad \lesssim |a_1-a_2|^{\alpha}|a_1|^{\alpha(1/\beta-1)} (|a_1^{1/\beta}-b^{1/\beta}|^{1-\alpha}+|a_2^{1/\beta}-b^{1/\beta}|^{1-\alpha})(|a_1^{1/\beta}-\bar{b}^{1/\beta}|+|b^{\frac{1}{\beta}}-a_2^{1/\beta}|).
\end{split}
\end{equation*} And then we see:
\begin{equation*}
\begin{split}
&\frac{|a_1|^{\frac{1}{\beta}-1}| (a_1^\frac{1}{\beta}-\bar{b}^\frac{1}{\beta})(a_1^\frac{1}{\beta}-b^\frac{1}{\beta})-(a_2^\frac{1}{\beta}-\bar{b}^\frac{1}{\beta})(a_2^\frac{1}{\beta}-b^\frac{1}{\beta})||b^\frac{1}{\beta}-\bar{b}^\frac{1}{\beta}|}{|a_1^\frac{1}{\beta}-\bar{b}^\frac{1}{\beta}||a_1^\frac{1}{\beta}-{b}^\frac{1}{\beta}||a_2^\frac{1}{\beta}-\bar{b}^\frac{1}{\beta}||a_2^\frac{1}{\beta}-{b}^\frac{1}{\beta}|} \\
&\quad\lesssim |a_1|^{(\alpha+1)(\frac{1}{\beta}-1)}|b^\frac{1}{\beta}-\bar{b}^\frac{1}{\beta}| \frac{|a_1-a_2|^{\alpha}(|a_1^{1/\beta}-b^{1/\beta}|^{1-\alpha}+|a_2^{1/\beta}-b^{1/\beta}|^{1-\alpha})(|a_1^{1/\beta}-\bar{b}^{1/\beta}|+|b^{\frac{1}{\beta}}-a_2^{1/\beta}|)}{|a_1^\frac{1}{\beta}-\bar{b}^\frac{1}{\beta}||a_1^\frac{1}{\beta}-{b}^\frac{1}{\beta}||a_2^\frac{1}{\beta}-\bar{b}^\frac{1}{\beta}||a_2^\frac{1}{\beta}-{b}^\frac{1}{\beta}|}
\end{split}
\end{equation*} which consists of four terms. A typical term is of the form: 
\begin{equation*}
\begin{split}
 &|a_1|^{(\alpha+1)(\frac{1}{\beta}-1)}|b^\frac{1}{\beta}-\bar{b}^\frac{1}{\beta}| \frac{|a_1-a_2|^{\alpha}|a_1^{1/\beta}-b^{1/\beta}|^{1-\alpha}|a_1^{1/\beta}-\bar{b}^{1/\beta}|}{|a_1^\frac{1}{\beta}-\bar{b}^\frac{1}{\beta}||a_1^\frac{1}{\beta}-{b}^\frac{1}{\beta}||a_2^\frac{1}{\beta}-\bar{b}^\frac{1}{\beta}||a_2^\frac{1}{\beta}-{b}^\frac{1}{\beta}|}\\
 &\quad=|a_1|^{(\alpha+1)(\frac{1}{\beta}-1)}|b^\frac{1}{\beta}-\bar{b}^\frac{1}{\beta}| \frac{|a_1-a_2|^{\alpha}}{|a_1^\frac{1}{\beta}-{b}^\frac{1}{\beta}|^\alpha|a_2^\frac{1}{\beta}-\bar{b}^\frac{1}{\beta}||a_2^\frac{1}{\beta}-{b}^\frac{1}{\beta}|}\\
 &\quad\leq |a_1|^{(\alpha+1)(\frac{1}{\beta}-1)}|a_1-a_2|^\alpha (\frac{1}{|a_1^\frac{1}{\beta}-{b}^\frac{1}{\beta}|^\alpha|a_2^\frac{1}{\beta}-{b}^\frac{1}{\beta}|} +\frac{1}{|a_1^\frac{1}{\beta}-{b}^\frac{1}{\beta}|^\alpha|a_2^\frac{1}{\beta}-\bar{b}^\frac{1}{\beta}|}),
\end{split}
\end{equation*}  using the Cauchy-Schwarz inequality. Now we notice:
$$\int_{\Omega_{\epsilon}} \frac{1}{|a_1^\frac{1}{\beta}-{b}^\frac{1}{\beta}|^\alpha|a_2^\frac{1}{\beta}-\bar{b}^\frac{1}{\beta}|}db\lesssim (|a_1|^{2}+1) (|a_1|^{-(\alpha+1)/\beta})$$ where we use that $\beta<\frac{1}{2}.$ 
Now collecting the terms we have estimated and similarly estimating the terms we have left out, we get:\begin{equation}\label{eq:key}
\begin{split}
\int_{\Omega_{\epsilon}} |K(a_1,b)-K(a_2,b)|db \lesssim |a_1-a_2|^{\alpha}|a_1|^{1-\alpha}.
\end{split}
\end{equation} This implies that $\frac{1}{1+|a|}\nabla\Delta^{-1}:L^\infty\rightarrow \mathring{C}^\alpha  $ for all $\alpha<1.$ 

\medskip
 
Now taking $\frac{1}{1+r}\partial_r\psi$ as a source term and using that $D^2\Delta^{-1}_{D}$ is a bounded operator on $\mathring{C}^{0,\alpha}(\Omega_\epsilon)$, we finally obtain that \begin{equation*}
\begin{split}
|D^2\Delta_D^{-1} \psi |_{\mathring{C}^\alpha} \le C|f|_{\mathcal{V}} 
\end{split}
\end{equation*} and this finishes the proof. 
\end{proof}

We now state the H\"older version of the previous lemma. 

\begin{lemma}\label{lem:apriori2}
	In addition to the assumptions of Lemma \ref{lem:apriori}, assume further that $f \in C^\alpha(\Omega_{\epsilon})$ and $0 < \alpha < 1/\beta - 2$ where $0 < \beta < 1/2$ with $\tan(\beta\pi) = \epsilon^{-1}$. Then, for the solution $\psi$ of $L\psi = f$, we have \begin{equation}\label{eq:Hoelder}
	\begin{split}
	|D^2\psi |_{C^\alpha(\Omega_{\epsilon})} \le C( |f|_{\mathcal{V}}  + |f|_{C^\alpha(\Omega_{\epsilon})})
	\end{split}
	\end{equation} with $C = C(\alpha,\epsilon) > 0$. 
\end{lemma}

To prove the above lemma, it is only necessarily to obtain the a priori bound \begin{equation*}
\begin{split}
\left| \frac{1}{1+r} \nabla\Delta^{-1} \psi \right|_{C^\alpha(\Omega_{\epsilon})} \le C|f|_{L^\infty},
\end{split}
\end{equation*} which follows readily from the estimate \eqref{eq:key}. We omit the details. Finally, as a corollary of Lemma \ref{lem:apriori2}, we obtain \begin{corollary}\label{cor:vanishing2}
	Under the assumptions of Lemma \ref{lem:apriori2}, if $f(0) = 0$, we have that \begin{equation*}
	\begin{split}
	|D^2\psi(x)| \le C_{\alpha,\epsilon} |x|^{\alpha} ( |f|_{\mathcal{V}}  + |f|_{C^\alpha(\Omega_{\epsilon})}).
	\end{split}
	\end{equation*}
\end{corollary} The proof is parallel to the one given for Corollary \ref{cor:vanishing} above.

\subsection{Further estimates on $\Omega_\epsilon$}

In this subsection we give some estimates which are useful for the global well-posedness of zero-swirl solutions as well as getting the blow-up criterion (of Beale-Kato-Majda type) in the general case.

\begin{lemma}\label{lem:improved}
	Let $\omega\in \mathring{C}^{0,\alpha}(A_\epsilon)$ be compactly supported. Let $\psi$ be the unique solution to $$\frac{1}{r} \partial_{zz} \psi + \frac{1}{r}\partial_{rr}\psi - \frac{1}{r^2}\partial_r\psi=\omega$$ on $A_\epsilon$ so that $\psi=0$ on $\partial A_{\epsilon}$ constructed using Lemma \ref{lem:apriori}. Then, there exists a constant $C_{\alpha,\epsilon}$ depending only on $\alpha$ and $\epsilon$ (but independent of the radius of the support of $\omega$) so that  \begin{equation}\label{eq:bkm1}
	\begin{split}
	|\frac{1}{1+|x|^2} \nabla\psi|_{L^\infty} \le C_{\alpha,\epsilon} \left( |\omega|_{L^1} + |\frac{\omega}{r}|_{L^\infty} + | \frac{\nabla\psi}{\sqrt{r}}|_{L^2} \right),
	\end{split}
	\end{equation} \begin{equation}\label{eq:bkm2}
	\begin{split}
	|\frac{1}{r} D^2\psi|_{\mathring{C}^{0,\alpha}}\leq C_{\alpha,\epsilon} \Big(|\frac{\omega}{r}|_{\mathring{C}^{0,\alpha}}+|\omega|_{L^1}+ |\frac{\nabla\psi}{\sqrt{r}}|_{L^2}\Big),
	\end{split}
	\end{equation} and \begin{equation}\label{eq:log2}
	\begin{split}
	|\frac{1}{r} D^2\psi|_{L^\infty}\leq C_{\alpha,\epsilon} \left( |\omega|_{L^1} + |\frac{\omega}{r}|_{L^\infty} + | \frac{\nabla\psi}{\sqrt{r}}|_{L^2} \right) \log\left( 2 + |\frac{\omega}{r}|_{\mathring{C}^{0,\alpha}} \right)
	\end{split}
	\end{equation} holds.
\end{lemma}

\begin{remark}
	Note that the term $|\frac{\nabla\psi}{\sqrt{r}}|_{L^2}$ is just $||v|\sqrt{r}|_{L^2}$ which is controlled by the total kinetic energy. 
\end{remark}

\begin{proof}
	In the proof we will suppress writing out the dependence of multiplicative constants on $\alpha, \epsilon$. 
	
	Recall that $\Delta\psi=r\omega+\frac{1}{r}\partial_r\psi.$ Thus, $$\nabla\psi=\int K(x,y) y_1 \omega(y)dy + \int K(x,y) \frac{1}{y_1}\partial_{y_1}\psi dy$$ with $|K(x,y)|\leq \frac{C}{|x-y|},$ so that $$|\nabla\psi|\leq C\left( \int \frac{|y_1|}{|x-y|}|\omega(y)|dy + \int \frac{|\partial_{y_1}\psi|}{|y_1||x-y|}dy\right).$$
	We estimate each term separately. Regarding the first integral, by writing $y_1 = (y_1 - x_1) + x_1$ and taking absolute values, we see that it is bounded by 
	\begin{equation*}
	\begin{split}
	|\omega|_{L^1} + |x| \int \frac{|y|}{|x-y|} \frac{|\omega(y)|}{|y|}dy. 
	\end{split}
	\end{equation*} Splitting the second integral into pieces $|x-y| \le 1$ and $|x-y| > 1$, we see that it is bounded by \begin{equation*}
	\begin{split}
	C(1 + |x|^2) \left( |\omega|_{L^1} + | \frac{\omega(y)}{|y|} |_{L^\infty} \right)
	\end{split}
	\end{equation*}
	
	Regarding the second term, $$\int \frac{|\partial_{y_1}\psi|}{|y_1||x-y|}= \int_{|x-y|<\delta}  \frac{|\partial_{y_1}\psi|}{|y_1||x-y|}dy + \int_{|x-y|>\delta}  \frac{|\partial_{y_1}\psi|}{|y_1||x-y|} dy,$$  where $\delta \le 1$ is a constant to be chosen later. In the region where $|x-y|<\delta$ we estimate:  $$\int_{|x-y|<\delta}  \frac{|\partial_{y_1}\psi|}{|y_1||x-y|}dy\leq \int_{|x-y|<\delta} \frac{|x-y| + |x|}{|x-y|}  \frac{|\partial_{y_1}\psi|}{|y|^2}dy \le  C\delta ( 1 + |x|) |\frac{\nabla\psi}{1 + |x|^2}|_{L^\infty}.$$ Next, we estimate $$ \int_{|x-y|>\delta}  \frac{|\partial_{y_1}\psi|}{|y_1||x-y|} dy\leq |\frac{\nabla\psi}{\sqrt{y_1}}|_{L^2}\sqrt{\int_{|x-y|>\delta} \frac{1}{|y_1||x-y|^2}}\leq C_\delta |\frac{\nabla\psi}{\sqrt{y_1}}|_{L^2}.$$
	
	Now putting all this together we get:\begin{equation*}
	\begin{split}
	|\nabla\psi(x)| \le C(1 + |x|^2) \left( |\omega|_{L^1} + | \frac{\omega(y)}{|y|} |_{L^\infty} \right) +  C\delta ( 1 + |x|) |\frac{\nabla\psi}{1 + |x|^2}|_{L^\infty} + C_\delta |\frac{\nabla\psi}{\sqrt{y_1}}|_{L^2}.
	\end{split}
	\end{equation*} Dividing both sides by $1 + |x|^2$ and choosing $\delta > 0$ to be a sufficiently small constant (possibly depending on $\alpha, \epsilon$), we conclude that 
	$$|\frac{\nabla\psi}{1+|x|^2}|_{L^\infty}\leq C( |\frac{\nabla\psi}{\sqrt{r}}|_{L^2}+|\omega|_{L^1}+|\frac{\omega}{r}|_{L^\infty}).$$ We now proceed as in Step 5 of the proof of Lemma \ref{lem:apriori} to conclude \eqref{eq:bkm2} and \eqref{eq:log2} from \eqref{eq:bkm1} and \eqref{eq:log}. 	\end{proof}

\section{Local well-posedness}\label{sec:lwp}

We are now ready to precisely state and prove the local well-posedness theorem in $\mathring{C}^\alpha(\Omega_{\epsilon})$ for the system \eqref{eq:3DEuler} -- \eqref{eq:elliptic}. For simplicity we shall assume that the initial data is compactly supported in space. 

\begin{theorem}\label{thm:LWP} Let $\epsilon>0$ and $0<\alpha<1$. 
For every $\omega_0$ and $u_0$ which are compactly supported in $\Omega_\epsilon$ and for which $\omega_0 \in \mathring{C}^{0,\alpha}(\Omega_\epsilon)$ and $\nabla u_0\in \mathring{C}^{0,\alpha}(\Omega_\epsilon),$ there exists a $T>0$ depending only on $| \omega_0 |_{\mathring{C}^{0,\alpha}}, |\nabla u_0|_{\mathring{C}^{0,\alpha}}$, and the radius of the support of $\omega_0$ so that there exists a unique solution $(\omega ,u)$ to the axi-symmetric 3D Euler system \eqref{eq:3DEuler}--\eqref{eq:elliptic} with $\omega, \nabla u\in C([0,T); \mathring{C}^{0,\alpha}(\Omega_\epsilon))$ and $(\omega, u)$ remain compactly supported for all $t\in[0,T)$. Finally, $(\omega,\nabla u)$ cannot be continued as compactly supported $\mathring{C}^{0,\alpha}$ functions past $T^*$ if and only if \begin{equation}\label{eq:bkm}
\begin{split}
\limsup_{t\rightarrow T^*} \int_0^{T^*} |\omega(t,\cdot)|_{L^\infty} +   |\nabla u(t,\cdot)|_{L^\infty}   dt =+\infty.
\end{split}
\end{equation}
\end{theorem}

\begin{remark}
	Before we proceed to the proof, we note that for a compactly supported vorticity $\omega(t,\cdot)$, we have the following estimate \begin{equation}\label{eq:vel_grad}
	\begin{split}
	|\nabla v |_{\mathring{C}^\alpha(\Omega_{\epsilon})} \le C_M |\omega|_{\mathring{C}^\alpha(\Omega_{\epsilon})}
	\end{split}
	\end{equation} as an immediate consequence of Lemma \ref{lem:apriori}, where $M > 0$ is the radius of the support of $\omega(t,\cdot)$. 
\end{remark}

\begin{proof}
	
To prove the local well-posedness result, we proceed in four steps: a priori estimates, existence, uniqueness, and lastly the blow-up criterion. In the following, 	we shall stick to the variables $(\eta,z)$ and use the system \eqref{eq:omega_eta} -- \eqref{eq:elliptic_eta}.  We begin with obtaining an appropriate set of a priori estimates for the axi-symmetric Euler system.

	\medskip

	\textit{(i) A priori estimates}

	\medskip
	
	We assume that there is a solution $(\omega,u)$ satisfying $u(0,\cdot) = u_0$, $\omega(0,\cdot) = \omega_0$ as well as $\omega(t,\cdot), \nabla u(t,\cdot) \in \mathring{C}^\alpha(\Omega_\epsilon)$ for some interval of time $[0,T)$. We further assume that the supports of $\omega$ and $u$ on $[0,T)$ are contained in a ball $B_0(M)$ for some fixed $M > 0$ (we shall suppress from writing out the dependence of multiplicative constants on $M$ momentarily). Then, from the elliptic estimate \eqref{eq:vel_grad}, we see in particular that the velocity $v$ is Lipschitz for $ 0 \le t < T$, and hence there exists a Lipschitz flow map $\Phi_t(\cdot) = \Phi(t,\cdot): \Omega_\epsilon \rightarrow \Omega_\epsilon$  with Lipschitz inverse $\Phi^{-1}_t$, defined by solving $\frac{d}{dt}\Phi(t,x) = v(t,\Phi(t,x))$ and $\Phi(0,x) = x$.  We have $v \cdot n = 0$ where $n$ is the outwards unit normal vector on $\partial\Omega_{\epsilon}$, and therefore $\Phi_t$ and $\Phi_t^{-1}$ maps $\partial\Omega_\epsilon$ to itself and fixes the origin. 
	
	First, evaluating the $u$-equation at $(\eta,z) = (0,0)$, we have that $u(t,0) = u_0(0)$ for all time. Next, we write the equations for $\omega$ and $\nabla u$ along the flow: we have  \begin{equation}\label{eq:omega_flow}
	\begin{split}
	\frac{d}{dt} (\omega\circ\Phi) = \left(\frac{2u\pr_z u + v_1\omega}{1+\eta}\right)\circ\Phi, 
	\end{split}
	\end{equation}\begin{equation}\label{eq:pr_etau_flow}
	\begin{split}
	\frac{d}{dt}\left( \pr_\eta u \circ\Phi \right) = \left(- \pr_\eta\left(\frac{v_1u}{1+\eta}\right) - \pr_\eta v_1 \pr_\eta u - \pr_\eta v_2 \pr_z u \right) \circ \Phi,
	\end{split}
	\end{equation} and \begin{equation}\label{eq:pr_zu_flow}
	\begin{split}
	\frac{d}{dt}\left( \pr_z u \circ\Phi\right) = \left( -  \frac{\pr_z(v_1u)}{1+\eta} - \pr_z v_1 \pr_\eta u - \pr_z v_2 \pr_z u \right)\circ\Phi.
	\end{split}
	\end{equation} From \eqref{eq:omega_flow}, we obtain \begin{equation*}
	\begin{split}
	\left| \frac{d}{dt} |\omega|_{L^\infty} \right| \le C| \nabla u |_{L^\infty}(1 + |\nabla u|_{L^\infty}) + |\nabla v|_{L^\infty} |\omega|_{L^\infty},
	\end{split}
	\end{equation*} where $C > 0 $ depends only on $u(t,0)$, which is constant in time. Proceeding similarly for \eqref{eq:pr_etau_flow} and \eqref{eq:pr_zu_flow}, we obtain  \begin{equation*}
	\begin{split}
	\left| \frac{d}{dt}|\nabla u|_{L^\infty}\right| \le C  |\nabla v|_{L^\infty}(1 + |\nabla u|_{L^\infty}) .
	\end{split}
	\end{equation*} Next, to obtain estimates for the $\mathring{C}^\alpha$-norm, one takes two points $x \ne x'$ and simply computes \begin{equation*}
	\begin{split}
	\frac{d}{dt} \left( \frac{|\Phi(t,x)|^\alpha \omega(\Phi(t,x)) - |\Phi(t,x')|^\alpha \omega(\Phi(t,x'))}{|\Phi(t,x) - \Phi(t,x')|^\alpha} \right),
	\end{split}
	\end{equation*} and similarly for $\pr_\eta u$ and $\pr_z u$. After straightforward computations (one may refer to the proof of \cite[Theorem 1]{EJB} for details), one obtains \begin{equation}\label{eq:apriori_circlealpha}
	\begin{split} 
	\frac{d}{dt} | \omega|_{ \mathring{C}^\alpha } &\le C\left( | \nabla v|_{L^\infty} | \omega|_{ \mathring{C}^\alpha } + (1 + | \nabla u|_{L^\infty}) | \nabla u|_{ \mathring{C}^\alpha } \right),\\
	\frac{d}{dt} | \nabla u|_{ \mathring{C}^\alpha } &\le C\left( | \nabla v|_{L^\infty} | \nabla u|_{ \mathring{C}^\alpha } + (1 + | \nabla u|_{L^\infty})| \nabla v|_{ \mathring{C}^\alpha }  \right).
	\end{split}
	\end{equation}
	 
	Given the above estimates, once we set  \begin{equation*}
	 	\begin{split}
	 	| \omega(t)|_{\mathring{C}^\alpha} + | \nabla u(t)|_{\mathring{C}^\alpha} = A(t)
	 	\end{split}
 	\end{equation*} on the time interval $[0,T)$, we have the inequality \begin{equation}\label{eq:apriori_calpha}
 	\begin{split}
 	\frac{d}{dt} A \le C(1 + A^2). 
 	\end{split}
 	\end{equation} Here, $C = C(M) > 0$ can be taken to be a continuous non-decreasing function of $M$. Let us set $$M(t) = \sup_{ \{x: \omega(t,x) \ne 0 \mbox{ or } u(t,x) \ne 0 \} } |x| ,$$ and then from \begin{equation*}
 	\begin{split}
 	\frac{d}{dt}\Phi(t,x) = v(t,\Phi(t,x)), 
 	\end{split}
 	\end{equation*} we see that \begin{equation}\label{eq:apriori_supp}
 	\begin{split}
 	\left|\frac{d}{dt} M(t) \right| \le  |\nabla v(t,\cdot)|_{L^\infty} M(t). 
 	\end{split}
 	\end{equation} Then, equations \eqref{eq:apriori_calpha} and \eqref{eq:apriori_supp} imply that there exists $T_1 = T_1(M(0), |\omega(0)|_{\mathring{C}^\alpha}, |\nabla u(0)|_{\mathring{C}^\alpha}, |u(0)|) > 0$ such that \begin{equation}\label{eq:T_1}
 	\begin{split}
 	\sup_{ [0,T_1] } \left( A(t) + M(t) \right) \le 2(A(0) + M(0)).
 	\end{split}
 	\end{equation} At this point, we have deduced (formally) that the solution can be continued past $T^* > 0$ if and only if \begin{equation*}
 	\begin{split}
 	\sup_{t \in [0,T^*]} \left( |\omega(t,\cdot)|_{\mathring{C}^\alpha} + |\nabla u(t,\cdot)|_{\mathring{C}^\alpha} + M(t) \right) < + \infty. 
 	\end{split}
 	\end{equation*}
	 

	\medskip

	\textit{(ii) Existence}

	\medskip

	We sketch a proof of existence based on a simple iteration scheme. Given the initial data $(\omega_0,u_0)$ with compact support, we take the time interval $[0,T_1]$ where $T_1 > 0$ is provided from the a priori estimates above (see \eqref{eq:T_1}). We shall define a sequence of approximate solutions $\{ (\omega^{(n)},u^{(n)},\Phi^{(n)}) \}_{n \ge 0}$ such that we have uniform bounds \begin{equation*}
	\begin{split}
	\sup_{t \in [0,T_1] } \left(A^{(n)}(t) + M^{(n)}(t) \right) \le C
	\end{split}
	\end{equation*} for all $n \ge 0$. Here, \begin{equation*}
	\begin{split}
	A^{(n)}(t) := |\omega^{(n)}(t)|_{\mathring{C}^\alpha} + |\nabla u^{(n)}(t)|_{\mathring{C}^\alpha}
	\end{split}
	\end{equation*} and \begin{equation*}
	\begin{split}
	M^{(n)}(t):=  \sup_{ \{x: \omega^{(n)}(t,x) \ne 0 \mbox{ or } u^{(n)}(t,x) \ne 0 \} } |x| ,
	\end{split}
	\end{equation*}

	The initial triple $(\omega^{(0)},u^{(0)},\Phi^{(0)})$ is defined by simply setting $\omega^{(0)} \equiv \omega_0$, $u^{(0)} \equiv u_0$, and letting $\Phi^{(0)}$ to be the flow associated with $v^{(0)}$ which is simply the time-independent velocity from $\omega_0$. Given $(\omega^{(n)},u^{(n)},\Phi^{(n)})$ for some $n \ge 0$, we define $(\omega^{(n+1)},u^{(n+1)},\Phi^{(n+1)})$ as follows: first, we obtain $u^{(n+1)}$ from solving the ODE \begin{equation*}
	\begin{split}
	\frac{d}{dt} u^{(n+1)}(t,\Phi^{(n)}(t,x)) = \left( - \frac{v_1^{(n)}(t,x) u^{(n+1)}(t,x)}{1+\eta} \right)\circ \Phi^{(n)}(t,x),
	\end{split}
	\end{equation*} for each $x \in \Omega_\epsilon$, and then $\omega^{(n+1)}$ is defined similarly as the solution of \begin{equation*}
	\begin{split}
	\frac{d}{dt} \omega^{(n+1)}(t,\Phi^{(n)}(t,x)) = \left( \frac{v_1^{(n)}(t,x)\omega^{(n+1)}(t,x) +2u^{(n+1)}(t,x)\pr_z u^{(n+1)}(t,x)}{1+\eta} \right)\circ \Phi^{(n)}(t,x),
	\end{split}
	\end{equation*} respectively on the time interval $[0,T_1]$. Here, $v^{(n)}$ is simply the associated velocity of $\omega^{(n)}$. It is straightforward to check that the sequence $\{ (\omega^{(n)},u^{(n)}) \}_{n\ge 0}$ satisfies the desired uniform bound, arguing along the lines of the proof of the a priori estimates above. 
	
	By passing to a subsequence, we have convergence $\omega^{(n)} \rightarrow \omega$ and $\nabla u^{(n)} \rightarrow \nabla u$ in $L^\infty([0,T_1];L^\infty(\Omega_{\epsilon}))$ for some functions $\omega, \nabla u$ belonging to $L^\infty([0,T_1]; \mathring{C}^\alpha(\Omega_\epsilon))$. It is then straightforward to check that $\omega, \nabla u$ is a solution with initial data $\omega_0, \nabla u_0$.

	\medskip

	\textit{(iii) Uniqueness}

	\medskip
	
	For the proof of uniqueness, we return to the velocity formulation of the axisymmetric Euler equations: \begin{equation*}
	\begin{split}
	\pr_t v + v\cdot\nabla v + \nabla p &= \frac{1}{r} \begin{pmatrix}
	u^2 \\ 0
	\end{pmatrix}, \\
	\pr_t u + v\cdot\nabla u &= -\frac{v_1 u}{r}
	\end{split}
	\end{equation*} supplemented with \begin{equation*}
	\begin{split}
	\mathrm{div}(rv) = 0. 
	\end{split}
	\end{equation*} We assume that for some time interval $[0,T]$, there exist two solutions $(v^{(1)},u^{(1)})$ and $(v^{(2)},u^{(2)})$ to the above system with the same initial data $(v_0,u_0)$. It is assumed that \begin{equation*}
	\begin{split}
	\sup_{t \in [0,T]} \left( |\nabla v^{(i)}|_{L^\infty} + |\nabla u^{(i)}|_{L^\infty} \right) \le C 
	\end{split}
	\end{equation*} for $i = 1, 2$. Setting \begin{equation*}
	\begin{split}
	V = v^{(1)} - v^{(2)},\qquad U = u^{(1)} - u^{(2)}, \qquad P = p^{(1)} - p^{(2)},
	\end{split}
	\end{equation*} we obtain \begin{equation*}
	\begin{split}
	\pr_t V + v^{(1)}\cdot\nabla V + V \cdot\nabla v^{(2)} + \nabla P = \frac{1}{r} \begin{pmatrix}
	U(u^{(1)} + u^{(2)})
	\end{pmatrix}
	\end{split}
	\end{equation*} and \begin{equation*}
	\begin{split}
	\pr_t U + v^{(1)}\cdot\nabla U + V\cdot\nabla u^{(2)} = -\frac{1}{r}\left( v_1^{(1)} U - V_1 u^{(2)}\right).
	\end{split}
	\end{equation*} Multiplying last two equations by $V$ and $U$ respectively and integrating against the measure $rdrdz$ on $\Omega_\epsilon$, we obtain after using the divergence-free condition that \begin{equation*}
	\begin{split}
	\frac{d}{dt} \mathcal{E}(t) \le C\left( 1 +  |\nabla u^{(1)}|_{L^\infty} + |\nabla u^{(2)}|_{L^\infty}+ |\nabla v^{(1)}|_{L^\infty} + |\nabla v^{(2)}|_{L^\infty} \right)\mathcal{E}(t)
	\end{split}
	\end{equation*} holds, with \begin{equation*}
	\begin{split}
	\mathcal{E}(t) := \left(\int_{\Omega_\epsilon}( U^2 + V^2) rdrdz\right)^{1/2}. 
	\end{split}
	\end{equation*} Since $\mathcal{E}(0) = 0$, we have $\mathcal{E} \equiv 0 $ on $[0,T]$. This finishes the proof of uniqueness.

	\medskip
	
	\textit{(iv) Blow-up criterion}
	
	\medskip
	
	We now use estimates from Lemma \ref{lem:improved} to establish the blow-up criterion \eqref{eq:bkm}. For this we assume that \begin{equation}\label{eq:bkmassume}
	\begin{split}
	\int_0^{T^*} |\omega(t,\cdot)|_{L^\infty} + |\nabla u(t,\cdot)|_{L^\infty} dt \le C. 
	\end{split}
	\end{equation} To begin with, from the equation for $u$, we see that \begin{equation*}
	\begin{split}
	|u(t,x)| \le C(1 + |x|)
	\end{split}
	\end{equation*} for all $t \ge 0$ with constant $C > 0$ depending only on the initial data. Using this observation on the equation for $\omega$, we see that \begin{equation*}
	\begin{split}
	\left|\frac{d}{dt} \left| \frac{\omega}{r} \right|_{L^\infty} \right| \le C \left| \frac{ u \pr_z u}{r^2} \right| \le C| \nabla u|_{L^\infty}
	\end{split}
	\end{equation*} and from \eqref{eq:bkmassume} we see \begin{equation*}
	\begin{split}
	\sup_{t \in [0,T^*]} \left| \frac{\omega(t,\cdot)}{r} \right|_{L^\infty} \le C. 
	\end{split}
	\end{equation*} On the other hand, \begin{equation*}
	\begin{split}
	\left|\frac{d}{dt} |\omega|_{L^1} \right|\le C \left| \frac{u \pr_z u}{r} \right|_{L^1} \le C|\nabla u|_{L^\infty} |ur^{1/2}|_{L^2}. 
	\end{split}
	\end{equation*} Here $  |ur^{1/2}|_{L^2}$ is bounded by the kinetic energy, which is finite for all time. Hence we also obtain that \begin{equation*}
	\begin{split}
	\sup_{t \in [0,T^*]}  | \omega(t,\cdot)|_{L^1} \le C. 
	\end{split}
	\end{equation*} Applying this on the estimates \eqref{eq:bkm1}, \eqref{eq:bkm2}, and \eqref{eq:log2}, we first see from $|v(x)| \le C(1 + |x|)$ (with $C > 0$ depending only on the initial data) that the size of the support remains uniformly bounded on the time interval $[0,T^*]$. Moreover, we obtain \begin{equation*}
	\begin{split}
	|\nabla v(t,\cdot)|_{L^\infty} \le C(1 + \log(2 + |\omega(t,\cdot)|_{\mathring{C}^\alpha}) )
	\end{split}
	\end{equation*} and \begin{equation*}
	\begin{split}
	|\nabla v(t,\cdot)|_{\mathring{C}^\alpha} \le C( 1 + |\omega(t,\cdot)|_{\mathring{C}^\alpha}).
	\end{split}
	\end{equation*} Then, returning to the proof of a priori estimates in the above we obtain this time the inequality \begin{equation*}
	\begin{split}
	\left|\frac{d}{dt} A \right| \le C(|\omega|_{L^\infty} + |\nabla u|_{L^\infty})(1 + A) \log(2 + A), 
	\end{split}
	\end{equation*} where $A(t) := |\omega(t,\cdot)|_{\mathring{C}^\alpha} + |\nabla u(t,\cdot)|_{\mathring{C}^\alpha}$. Using Gronwall's inequality, it is straightforward to show that $\sup_{t \in [0,T^*]} A \le C$. The proof is now complete. \end{proof}

\subsection{Global regularity in the no-swirl case}

Here we prove that the unique local solution we have constructed above is global in the no-swirl case, i.e. when $u \equiv 0$. In this case, the system simply reduces to \begin{equation}\label{eq:noswirl}
\begin{split}
\pr_t \left( \frac{\omega}{r} \right) + v\cdot\nabla  \left( \frac{\omega}{r} \right) =0 . 
\end{split}
\end{equation}

\begin{proof}[Proof of Theorem \ref{MainThm3}]
	We first note from \eqref{eq:noswirl} that $|\omega/r|_{L^\infty}$ and $|\omega|_{L^1}$ are both conserved in time. The latter follows from the fact that $\mathrm{div}(rv) = 0$. From the estimate  \eqref{eq:bkm1} it follows that $|v(t,x)| \le Cr $ for all time, where $C > 0$ only depends on $\alpha, \epsilon$ and the initial data $\omega_0$. From \begin{equation*}
	\begin{split}
	\frac{d}{dt} (\omega\circ\Phi) = \left( \frac{v_1\omega}{r} \right) \circ \Phi , 
	\end{split}
	\end{equation*} we obtain that \begin{equation*}
	\begin{split}
	|\omega(t,\cdot)|_{L^\infty} \le |\omega_0|_{L^\infty} \exp(Ct). 
	\end{split}
	\end{equation*} From the blow-up criterion \eqref{eq:bkm}, it follows that the solution is global in time. We note that using the estimate \eqref{eq:log2} we have $$|\nabla v(t,\cdot)|_{L^\infty} \le C(1 + \log(2 + |\omega(t,\cdot)|_{\mathring{C}^\alpha}))$$ and from this it is easy to deduce that the norm $|\omega(t,\cdot)|_{\mathring{C}^\alpha} $ can grow at most double exponentially in time. 
\end{proof}

\section{Proof of Blow-Up}\label{sec:blowup}

Here we show that the local solutions we constructed can actually become singular in finite time. As in Section \ref{sec:heuristic}, we write $r=\eta+1$ and let $R^2=\eta^2+z^2$ and $\theta=\arctan(\frac{z}{\eta})$. Then, in $(R,\theta)$-coordinates, the domain $\Omega_{\epsilon}$ is a sector $\{ (R,\theta) : R \ge 0, 0 \le \theta \le l \}$ with $l =\tan^{-1}(\epsilon^{-1})< \pi/2$. 

\begin{theorem}\label{thm:blowup}  Take any smooth initial data $g_0, P_0 \in C^\infty([0,l])$ whose local solution to the Boussinesq system for radially homogeneous data \eqref{eq:B1}--\eqref{eq:B2} blows up in finite time, and let $\phi \in C^\infty(\Omega_{\epsilon})$ be a radial cut-off function with $\phi(R)=1$ in $R\leq 1$ and $\phi(R)=0$ for $R\geq 2$. Then, the unique local solution to the axi-symmetric $3D$ Euler system \eqref{eq:3DEuler}--\eqref{eq:elliptic} provided by Theorem \ref{thm:LWP} corresponding to the initial data \begin{equation}\label{eq:id}
	\begin{split}
	\omega_0(R,\theta) &:=   g_0(\theta)\phi(R) \\
	u_0(R,\theta) &:= \left(\frac{1}{1+\eta}+R  P_0(\theta) \right)\phi(R)
	\end{split}
	\end{equation} blows up in finite time in the class $\omega(t,\cdot), \nabla u(t,\cdot) \in \mathring{C}^\alpha(\Omega_{\epsilon})$ for any $0 < \alpha <1$. 
\end{theorem}

\begin{remark}
	As we have discussed already in the introduction, from  our previous work on $2D$ Boussinesq system \cite{EJB}, we may take $g_0=0$ and $P_0=\theta^2$ and $g$ and $P$ will become singular at some $T^*<+\infty$.
\end{remark}

\begin{proof} 
For simplicity we shall assume that $|\phi'|,|\phi''|,|\phi'''|\leq 100$. Moreover, we just assume that $g_0,P_0\in C^{2,\alpha}([0,l])$ in \eqref{eq:id}. Note first that $\omega_0$ and $u_0$ defined in \eqref{eq:id} are compactly supported and $\omega,\nabla u\in \mathring{C}^{0,\alpha}$ for all $0\leq \alpha\leq 1$. By local well-posedness, we know that there is a unique local solution $\omega, u.$ Let us assume, towards a contradiction, that this solution is global.
Next, let's define $g(t,\theta)$ and $P(t,\theta)$ to be the unique solutions to the following system: \begin{equation}\label{eq:B1}
\begin{split}
\partial_t g+2G\partial_\theta g=2(\sin(\theta)P+\cos(\theta)\partial_\theta P),
\end{split}
\end{equation}\begin{equation}\label{eq:B2}
\begin{split}
\partial_t P+2G\partial_\theta P=\partial_\theta G P,
\end{split}
\end{equation}  with $G$ the unique\footnote{Note that $\arctan{\frac{1}{\epsilon}}<\frac{\pi}{2}$ for every $\epsilon>0$. } solution to the elliptic boundary value problem $$\partial_{\theta\theta} G+4G=g,\qquad G(0)=G(\tan^{-1}(\epsilon^{-1}))=0.$$ Let us take $T^* < +\infty$ to be the blow-up time for $(g,P)$ before which they retain initial smoothness. We will now show that for all $t<T^*$, $$\lim_{(r,z)\rightarrow (1,0)}|\omega - g\phi|=0.$$ This will imply that $\omega$ and $u$ must become singular in finite time. To see this through, we must first study the equation for $\tilde\omega:=\omega -g\phi$ and $\tilde u:=u -(RP + \frac{1}{1+\eta})\phi$. Notice that $\tilde\omega_0,\nabla \tilde u_0\in C^\alpha$ for every $0\leq \alpha\leq 1$ and that $\tilde\omega_0(0)=\nabla\tilde u_0(0)=0$ we want to propagate this at least for some $\alpha>0$.  
\subsubsection{The equations for the error terms}
Now let us see what evolution equations $\tilde\omega$ and $\tilde u$ satisfy. We begin with $\tilde \omega$. Recall first that $\omega$ satisfies:
$$\partial_t \frac{\omega}{\eta+1}+v\cdot\nabla \frac{\omega}{\eta+1}=\frac{1}{(\eta+1)^2}\partial_z(u^2)$$
$$v=\frac{1}{\eta+1}\nabla^\perp\psi$$ with $$\tilde{L}\psi := \frac{1}{(\eta+1)^2} L\psi   =\frac{1}{(\eta+1)^2}(\partial_{\eta\eta}\psi-\frac{1}{\eta+1}\partial_\eta\psi+\partial_{zz}\psi)=\omega $$ and $\psi=0$ on $\partial\Omega_\epsilon=\{0\leq \epsilon z\leq \eta\}$.  Next we substitute:
$$\omega=\tilde\omega+g\phi$$ and $$u=\tilde u+\phi(\frac{1}{\eta+1}+RP).$$
Direct substitution gives:
$$\frac{D}{Dt}\Big(\frac{\tilde\omega+g\phi}{\eta+1}\Big)=2\frac{\partial_z(\tilde u+\frac{\phi}{1+\eta}+RP\phi)(\tilde u+\frac{\phi}{1+\eta}+RP\phi)}{(1+\eta)^2}.$$
We set \begin{equation*}\begin{split} I:=&2\frac{(\partial_z\tilde u+\partial_z(RP\phi))(\tilde u+RP\phi)}{(1+\eta)^2}+\frac{\partial_z\tilde u \phi}{(1+\eta)^3}+\frac{RP\phi\partial_z\phi}{(1+\eta)^2}+\frac{\phi^2}{(1+\eta)^3}\partial_z(RP) \\ 
&\qquad-\frac{\phi}{(1+\eta)}\partial_z(RP)-g\frac{D}{Dt}\Big(\frac{\phi}{\eta+1}\Big) +\frac{2\pr_z \phi}{(1+\eta)^3}(\tilde{u} + \frac{\phi}{1+\eta} + RP\phi)\end{split}\end{equation*}
so that the equation for $\tilde{\omega}$ becomes $$\frac{D}{Dt} \Big(\frac{\tilde\omega}{\eta+1}\Big)+\frac{\phi}{\eta+1}\frac{Dg}{Dt}=I+\frac{\phi}{\eta+1}\partial_z(RP).$$
Now notice:
$$\frac{Dg}{Dt}=\partial_t g+v\cdot\nabla g.$$
Define $$\overline{v_g}:=\nabla^\perp(R^2G)$$ where $4G+G''=g$ as above. 
Then, $$\frac{Dg}{Dt}=\partial_t g+\overline{v_g}\cdot\nabla g+(v-\overline{v_g})\cdot\nabla g.$$ Note now that $\partial_t g+\overline{v_g}\cdot\nabla g=\partial_z (RP)$ (this is simply \eqref{eq:B1}). Thus, $$\frac{D}{Dt}\Big(\frac{\tilde\omega}{\eta+1}\Big)=I-\frac{\phi}{\eta+1}(v-\overline{v_g})\cdot\nabla g.$$
We will soon show that all the terms in $I$ and $(v-\overline{v_g})\cdot\nabla g$ belong to $C^\alpha(\Omega_\epsilon)$ and also vanish at $0$. The latter, in turn, will use Corollary \ref{cor:vanishing2} which is the crucial component in the proof besides observing that the 1D system governing $g$ and $P$ is the effective system near the origin. Before going through the estimates for $I$ and $(v-\overline{v_g})\cdot\nabla g$, let us perform a similar calculation to get the equation for $\tilde u$. 

Recall that $$\frac{D}{Dt}\Big((\eta+1)u\Big)=0.$$ Now we substitute $u=\tilde u+\phi\big(\frac{1}{1+\eta}+RP\big)$ and we get: $$\frac{D}{Dt}\Big(\tilde u(\eta+1)\Big)+\frac{D}{Dt}\Big(\phi+RP(\eta+1)\phi\Big)=0.$$
Next, we expand:
$$\frac{D}{Dt}\Big(\tilde u(\eta+1)\Big)=-\frac{D\phi}{Dt}-\frac{D}{Dt}(RP)\phi(\eta+1)-RP\frac{D}{Dt}(\phi(\eta+1))$$
As before, let's note:

$$\frac{D}{Dt}(RP)=R\partial_t P+v\cdot\nabla(RP)$$
Recalling $\overline{v_g}=\nabla^\perp(R^2G)$, we write: $$\frac{D}{Dt}(RP)=R\partial_t P+\overline{v_g}\cdot\nabla(RP)+(v-\overline{v_g})\cdot\nabla(RP).$$ Since $$R\partial_t P+\overline{v_g}\cdot\nabla(RP)=0,$$ (this is just the same as \eqref{eq:B2}) we obtain $$\frac{D}{Dt}\Big(\tilde u(\eta+1)\Big)=-\frac{D\phi}{Dt}-RP\frac{D}{Dt}(\phi(\eta+1))-(v-\overline{v_g})\cdot\nabla(RP)(\eta+1)\phi.$$

Thus, to this point we have that $\tilde\omega$ and $\tilde u$ satisfy:
$$\frac{D}{Dt}\Big(\frac{\tilde\omega}{\eta+1}\Big)=I-\frac{\phi}{\eta+1}(v-\overline{v_g})\cdot\nabla g,$$
$$\frac{D}{Dt}\Big(\tilde u(\eta+1)\Big)=II-(v-\overline{v_g})\cdot\nabla(RP)(\eta+1)\phi,$$
where we define $$II:=-\frac{D\phi}{Dt}-RP\frac{D}{Dt}(\phi(\eta+1)).$$ We have left the transport terms out of the definitions of $I$ and $II$ in order to emphasize their difference with $I$ and $II$. Indeed, those terms are where the non-locality of the problem is clearest. 

\subsubsection{Estimates on $I$ and $II$}
We now proceed to prove the following a-priori $C^\alpha$ estimates on $I$ and $C^{1,\alpha}$ estimates on $II$. 

\emph{Claim:} For $t<T^*$, \begin{equation} \label{eq:I_estimate}|I|_{C^\alpha}\lesssim 1+|\tilde\omega|_{C^\alpha}+|\tilde u|_{C^{1,\alpha}} \end{equation} and \begin{equation}\label{eq:II_estimate} |II|_{C^{1,\alpha}}\lesssim 1+|\tilde\omega|_{C^\alpha}+|\tilde u|_{C^{1,\alpha}}.\end{equation}
Note that the implicit constants may depend upon the $\mathring{C}^{0,\alpha}$ norms of $\omega$ and $\nabla u$ (which are assumed to be finite for all $t\geq 0$) as well as the $C^{2,\alpha}([0,l])$ norms of $g$ and $P$ and this is why we restrict our attention to $t<T^*$. 
Note also that we cannot allow for quadratic terms on the right hand side in $|\tilde u|_{C^{1,\alpha}}$ or $|\tilde\omega|_{C^\alpha}$ since we wouldn't be able to rule out $\tilde u$ becoming singular before $T^*$ which would just mean that our decomposition ceased to be valid. However, we cannot avoid terms like $|\tilde u|_{C^{1,\alpha}}|\nabla \tilde u|_{L^\infty},$ for example. In this case, we simply write: $|\nabla \tilde u|_{L^\infty}\leq |\nabla u|_{L^\infty}+|\nabla(\phi( \frac{1}{1+\eta}+RP))|_{L^\infty}$ which are both bounded for all $t<T^*$, the former since we assume there is no blow-up at all and the latter since $(g,P)$ does not blow-up until $T^*$. 

We will now proceed to prove the claim.  
Recall that 
\begin{equation*}\begin{split} I =&2\frac{(\partial_z\tilde u+\partial_z(RP\phi))(\tilde u+RP\phi)}{(1+\eta)^2}+\frac{\partial_z\tilde u \phi}{(1+\eta)^3}+\frac{RP\phi\partial_z\phi}{(1+\eta)^2}+\frac{\phi^2}{(1+\eta)^3}\partial_z(RP) \\ 
&\qquad-\frac{\phi}{(1+\eta)}\partial_z(RP)-g\frac{D}{Dt}\Big(\frac{\phi}{\eta+1}\Big) +\frac{2\pr_z \phi}{(1+\eta)^3}(\tilde{u} + \frac{\phi}{1+\eta} + RP\phi).\end{split}\end{equation*}
The only terms which may be dangerous are ones in which $P$ appears since all other functions are localized and belong to $C^\alpha$, which is an algebra. We will only give the details for those potentially dangerous terms. First observe that, $$|\tilde u|\leq |\nabla \tilde u|_{C^\alpha}|x|^{1+\alpha}$$ since $\tilde u(0)=\nabla\tilde u(0)=0$. %
Next $$|\frac{1}{(1+\eta)^2}\partial_z(RP\phi)\tilde u|_{C^\alpha}\leq|\partial_z(RP) \frac{\tilde u\phi}{(1+\eta)^2}|_{C^\alpha}+|\partial_z\phi \frac{RP\tilde u}{(1+\eta)^2}|_{C^\alpha}
$$ $$\leq C|P|_{C^{1,\alpha}([0,l])}|\nabla \tilde u|_{C^\alpha}$$ for some universal constant $C$.
Next, $$|\frac{\phi}{(1+\eta)^3}\partial_{z}(RP)-\frac{\phi}{(1+\eta)}\partial_z(RP)|_{C^\alpha}=|\phi\partial_{z}(RP) \frac{\eta(2+\eta)}{(1+\eta)^3}|_{C^\alpha}$$ 
$$\leq C|\phi\partial_{z}(RP)\eta|_{C^\alpha}\leq C|P|_{C^{1,\alpha}([0,l])}.$$
Next recall the following simple fact: if $f\in \mathring{C}^{0,\alpha}$ and $g\in C^\alpha$ satisfies $g(0)=0$, then $|fg|_{C^\alpha}\leq 2|f|_{\mathring{C}^{0,\alpha}}|g|_{C^\alpha}.$ Now we turn to $$\Big|g\frac{D}{Dt}\Big(\frac{\phi}{\eta+1}\Big)\Big|_{C^\alpha}=\Big|g v\cdot\nabla \frac{\phi}{\eta+1}\Big|_{C^\alpha}\leq C|g|_{C^\alpha([0,l])}|\omega|_{\mathring{C}^{0,\alpha}}.$$ The rest of the terms in $I$ can be handled similarly and we leave the details to the interested reader. 
Next, we turn to $II$, which must be controlled in $C^{1,\alpha}$ this time:
$$II=-\frac{D\phi}{Dt}-RP\frac{D}{Dt}(\phi(\eta+1)).$$
Both of these terms are handled similarly so we focus on the most difficult part: $$|RP\phi\frac{D}{Dt}(\eta+1)|_{C^{1,\alpha}}=|RP\phi v_2|_{C^{1,\alpha}}.$$ Note that neither $RP$ nor $v_2$ belong to $C^{1,\alpha}$ but their product does. Indeed, first we compute the gradient:
$$ |\nabla(RP\phi v_2)|_{C^{\alpha}} \leq |\nabla(RP)\phi v_2|_{C^\alpha}+|RP\nabla\phi v_2|_{C^\alpha}+|RP\phi\nabla v_2|_{C^\alpha}$$ 
$$\leq C|P|_{C^{1,\alpha}([0,l])}|\nabla v|_{\mathring{C}^{0,\alpha}}\leq C|P|_{C^{1,\alpha}([0,l])}|\omega|_{\mathring{C}^{0,\alpha}}.$$
The $C^{1,\alpha}$ estimate for $\frac{D\phi}{Dt}$ is bounded similarly. This completes the proof of the claim. 

\subsubsection{Estimates on the transport terms}
We now move to control the transport terms.

\medskip

\emph{Claim:}
$$|\frac{\phi}{\eta+1}(v-\overline{v_g})\cdot\nabla g|_{C^\alpha}\lesssim 1+|\tilde\omega|_{C^\alpha}$$
and
$$|(v-\overline{v_g})\cdot\nabla(RP)(\eta+1)\phi|_{C^{1,\alpha}}\lesssim 1+|\tilde\omega|_{C^{\alpha}},$$ where the implicit constants may depend on $P$ and $g$.
 
Notice, as in the above, that the $\nabla g$ in the first inequality is homogeneous of degree $-1$ so we will need to show that $v-\overline{v_g}$ vanishes to order $R^{1+\alpha}$ as $R\rightarrow 0$ to prove the $C^\alpha$ estimate. In fact, the proof of both inequalities follows from the following:
\begin{equation}\label{eq:error_vanishing}|\phi(v-\overline{v_g})|_{C^{1,\alpha}}+\frac{|\phi(\eta,z)(v(\eta,z)-\overline{v_g}(\eta,z))|}{R^{1+\alpha}}\lesssim 1+|\tilde\omega|_{C^\alpha}\end{equation} for all $\eta, z.$

To prove this inequality, we need to use Corollary \ref{cor:vanishing2}. Recall that $v=\nabla^{\perp}\tilde{L}^{-1}\omega$ and $\overline{v_g}=\nabla^{\perp}(R^2G).$ 
Let's compute $\tilde{L}(\phi(\tilde{L}^{-1}\omega-R^2G)):$

$$ \tilde{L}(\phi(\tilde{L}^{-1}\omega-R^2G))=\phi \tilde{L}(\tilde{L}^{-1}\omega-R^2G)+\tilde{L}(\phi)(\tilde{L}^{-1}\omega-R^2G)+\frac{2}{(1+\eta)^2}\nabla\phi\cdot\nabla(\tilde{L}^{-1}\omega-R^2G)$$
The main term is the first one:
$$\tilde{L}(\tilde{L}^{-1}\omega-R^2G) =\omega-\frac{1}{\eta+1}\partial_\eta(\frac{1}{\eta+1}\partial_{\eta}(R^2 G))-\frac{1}{(1+\eta)^2}\partial_{zz}(R^2G)$$
$$=\omega-\frac{1}{(\eta+1)^2}(\partial_{\eta\eta}(R^2G)+\partial_{zz}(R^2G))+\frac{1}{(\eta+1)^3}\partial_{\eta}(R^2G)$$
$$=\tilde\omega +g(\phi-\frac{1}{(\eta+1)^2})+\frac{1}{(\eta+1)^3}\partial_\eta (R^2G).$$
Thus, collecting all the terms, we get:

$$\tilde{L}(\phi(\tilde{L}^{-1}\omega- R^2G ))=\phi\tilde\omega+\phi(g(\phi-\frac{1}{(\eta+1)^2})+\frac{1}{(\eta+1)^3}\partial_\eta (R^2G))$$ $$+\tilde{L}(\phi)(\tilde{L}^{-1}\omega-R^2G)+\frac{2}{(1+\eta)^2}\nabla\phi\cdot\nabla(\tilde{L}^{-1}\omega-R^2G).$$
Thus, it is easy to see that
$\tilde{L}(\phi(\tilde{L}^{-1}\omega-R^2G))$ is compactly supported, $C^\alpha$, and vanishing at $0$. Moreover,
$$|\tilde{L}( \phi(\tilde{L}^{-1}\omega-R^2G) )|_{C^\alpha}\lesssim 1+|\tilde\omega|_{C^\alpha}$$ where the constant depends upon $|\omega|_{\mathring{C}^{0,\alpha}}$ and $|g|_{C^{\alpha}([0,l])}.$ Noting that the function $\phi(\tilde{L}^{-1}\omega-R^2G)$ vanishes on the boundary of $\Omega_{\epsilon}$, the claim then follows easily from Corollary \ref{cor:vanishing2}. 

\subsubsection{Closing the a-priori estimates on the error terms}
Now we will combine the estimates from the preceding subsections to close and estimate on $\tilde\omega$ in $C^\alpha$ and $\tilde u$ in $C^{1,\alpha}$. 
Collecting those estimates, we see:

\begin{equation}\label{eq:combined_schauder} \Big|\frac{D}{Dt}\Big(\frac{\tilde\omega}{\eta+1}\Big)\Big|_{C^{\alpha}}+\Big|\frac{D}{Dt}\Big((\eta+1)\tilde u\Big)\Big|_{C^{1,\alpha}}\lesssim (1+|\tilde\omega|_{C^\alpha}+|\tilde u|_{C^{1,\alpha}})\end{equation} where the implicit constant depends on $|\omega|_{\mathring{C}^{0,\alpha}}, |\nabla u|_{\mathring{C}^{0,\alpha}}, |g|_{C^{1,\alpha}([0,l])},$ and $|P|_{C^{2,\alpha}([0,l])}$ which are all finite for $t<T^*$.
Let $\Phi$ be the Lagrangian flow map associated to the full transport velocity field $v$. By assumption, the Lipschitz norm of $v$ is bounded for all time since $|v|_{Lip}\lesssim |\omega|_{\mathring{C}^{0,\alpha}}.$
Now let $\Phi$ be the Lagrangian flow map associated to $v$:
$$\frac{d}{dt}\Phi=v\circ\Phi,$$
$$\Phi|_{t=0}=Id.$$
It is easy to show that $$|\nabla\Phi|_{L^\infty}+|\nabla\Phi^{-1}|\leq \exp(\int_0^t|\nabla v(s,\cdot)|_{L^\infty}ds).$$
Next, notice that $$\partial_t(\frac{\tilde\omega}{\eta+1}\circ\Phi)=\Big[\frac{D}{Dt}\Big(\frac{\tilde\omega}{\eta+1}\Big)\Big]\circ\Phi.$$
Now using \eqref{eq:combined_schauder} and the Lipschitz bound on $\Phi$, we see:
$$\frac{d}{dt} |\tilde\omega|_{C^\alpha}\lesssim (1+|\tilde\omega|_{C^\alpha}+|\tilde u|_{C^{1,\alpha}}).$$
Now we move to prove a similar bound for $|\tilde u|_{C^{1,\alpha}}.$ We proceed in a completely analagous fashion except that we first differentiate $\frac{D}{Dt}((\eta+1)\tilde u)$ in space. 
Indeed, if $\partial$ denotes any spatial derivative (either in $\eta$ or $z$)
$$\frac{D}{Dt}( (\eta+1)\partial \tilde u)=\partial\frac{D}{Dt}((\eta+1)\tilde u)-\frac{D}{Dt}(\tilde u\partial \eta)-\partial v\cdot\nabla ((\eta+1)\tilde u).$$ 
As above, by using the flow map of $\Phi$ it suffices to show that the right hand side above can be bounded in $C^\alpha$. The first term belongs to $C^{\alpha}$ and its $C^\alpha$ norm is bounded using \eqref{eq:combined_schauder}. Next, the last term $\partial v\cdot\nabla (\eta+1)\tilde u$ can be controlled in the following way:
$$|\partial v\cdot\nabla((\eta+1)\tilde u)|_{C^\alpha}\lesssim |\partial v|_{\mathring{C}^{0,\alpha}}|\nabla ((\eta+1)\tilde u)|_{C^\alpha}\lesssim |\tilde u|_{C^{1,\alpha}}.$$
the first inequality being since $\nabla\tilde u(0)=0$. 
Now we turn to estimating the term $\frac{D}{Dt}(\tilde u\partial\eta).$ Notice that $\partial\eta\equiv 1$ or $\partial\eta\equiv 0$ depending on what kind of derivative $\partial$ is. Let us look at the former case. Then, 
$$\frac{D}{Dt}\tilde u=\partial_t\tilde u+v\cdot\nabla\tilde u.$$
It is easy to see that both of these terms are bounded in $C^\alpha$ by $|\tilde u|_{C^{1,\alpha}}$ multiplied by a constant (again, which may depend on $|\omega|_{\mathring{C}^{0,\alpha}}$) since $\frac{D}{Dt}((\eta+1)\tilde u)$ is already known to be in $C^\alpha$. We omit the details. \
Thus we get, finally: 
$$\frac{d}{dt} (|\tilde\omega|_{C^\alpha}+|\tilde u|_{C^{1,\alpha}})\lesssim  1+|\tilde\omega|_{C^\alpha}+|\tilde u|_{C^{1,\alpha}} .$$
Thus, so long as $t<T^*$, and so long as $\omega, \nabla u$ remain bounded for $t<T^*$, $$\tilde\omega=\omega-g\phi\in C^{\alpha}$$ and $\tilde\omega(0)=0$. 
Thus, for $t<T^*$, $$|\omega|_{L^\infty}\geq \limsup_{(\eta,z)\rightarrow 0}|\omega(\eta,z)|=\limsup_{(\eta,z)\rightarrow 0}|\tilde\omega +g\phi|=|g|_{L^\infty}.$$
However, we know that $\limsup_{t\rightarrow T^*}|g|_{L^\infty}=+\infty.$ This is a contradiction.
\end{proof}

\section{Futher Results}

We close this paper with a question regarding the $3D$ Euler equation on the domains $\Lambda_c:=\{(x,y,z): 0\leq z\leq c\sqrt{x^2+y^2}\}.$ A direct calculation just assuming that the velocity field is $1-$homogeneous in space and axi-symmetric on $\Lambda_c$ gives the system: \begin{equation}\label{eq:1D}
\begin{split}
&\pr_t g - 3G\pr_\theta g  = \left(G' + 2\tan\theta G \right)g + 2 \left( \tan\theta P + P' \right)P, \\
&\pr_t P - 3G\pr_\theta P  =  -\left(  2G' + \tan\theta G\right) P,
\end{split}
\end{equation} supplemented with \begin{equation}\label{eq:1D_G}
\begin{split}
6G - \left(\tan\theta G\right)' + G'' = g,\qquad G(0) = G(\beta\pi) = 0, 
\end{split}
\end{equation}
with $\beta$ depending on $c$ sufficiently small. Are smooth solutions to this $1D$ system global? If singularity formation is possible, then finite energy $W^{1,\infty}$ blow-up solutions can be found on $\Lambda_c$. Establishing blow-up for this $1D$ system seems to be more challenging than for the system \eqref{eq:B1}--\eqref{eq:B2}. Global regularity for this system would also be very interesting since it would indicate a regularizing mechanism at the axis in the axi-symmetric Euler equation (since we have established blow-up away from the axis in this paper).

\bibliographystyle{plain}
\bibliography{3dEuler_away_final}

\end{document}